\documentclass{amsart}
\usepackage[margin=1.5in]{geometry}
\usepackage{mathabx}
\usepackage{amsfonts}
\usepackage{lineno,hyperref}
\usepackage{amsmath}
\usepackage{amssymb}
\usepackage{amsthm}
\usepackage{cancel}
\usepackage{tikz-cd}
\usepackage{arydshln}
\usepackage{graphicx}

\newcommand{\RR}{\mathbb{R}}
\newcommand{\CC}{\mathbb{C}}

\newcommand{\ZZ}{\mathbb{Z}}
\newcommand{\QQ}{\mathbb{Q}}
\newcommand{\OO}{\mathcal{O}}
\newcommand{\HH}{\mathbb{H}}

\newcommand{\nrm}{\text{nrm}}
\newcommand{\tr}{\text{tr}}
\newcommand{\disc}{\text{disc}}

\newcommand{\Mat}{\text{Mat}}
\newcommand{\Isom}{\text{Isom}}

\theoremstyle{plain}
\newtheorem{definition}{Definition}[section]
\newtheorem{theorem}{Theorem}[section]
\newtheorem{corollary}{Corollary}[section]
\newtheorem{lemma}{Lemma}[section]

\theoremstyle{remark}
\newtheorem{example}{Example}
\newtheorem{remark}{Remark}[section]

\begin{document}
\title{Algebraic Groups Constructed From Rings with Involution}
\author{Arseniy (Senia) Sheydvasser}
\address{Department of Mathematics, Graduate Center at CUNY, 365 5th Ave, New York, NY 10016}
\email{ssheydvasser@gc.cuny.edu}


\date{\today}


\begin{abstract}
We define a class of groups constructed from rings equipped with an involution. We show that under suitable conditions, these groups are either algebraic or arithmetic, including as special cases the orientation-preserving isometry group of hyperbolic 4-space, $SL(2,R)$ for any commutative ring $R$, various symplectic and orthogonal groups, and an important class of arithmetic subgroups of $SO^+(4,1)$. We investigate when such groups are isomorphic and conjugate, and relate this to problem of determining when hyperbolic $4$-orbifolds are homotopic.
\end{abstract}

\maketitle

\section{Introduction:}
The idea of representing elements of $\text{M\"{o}b}(\RR^n)$ with $2\times 2$ matrices with entries in a Clifford algebra goes back at least to Vahlen \cite{Vahlen1902}, and was later popularized by Ahlfors \cite{Ahlfors1986}. More recently, this approach was used by the author to construct explicit examples of integral, crystallographic sphere packings \cite{Sheydvasser2019}; briefly, these are generalizations of the classical Apollonian gasket which arise from hyperbolic lattices. Such packings were formally defined by Kontorovich and Nakamura \cite{KontorovichNakamura2019}, although they were studied in various forms previously \cite{GLMWY2005,GuettlerMallows2010,Stange2015}. How to define $\text{M\"{o}b}(\RR)$ and $\text{M\"{o}b}(\RR^2)$ in terms of the real and complex matrices is well-known. In order to describe $\text{M\"{o}b}(\RR^3)$ in terms of $2 \times 2$ matrices, we proceed as follows: let $H_\mathbb{R}$ be the standard Hamilton quaternions, and define an involution $(x + yi + zj + tk)^\ddagger = x + yi + zj - tk$. One can then define the set
	\begin{align*}
	SL^\ddagger(2,H_\mathbb{R}) = \left\{\begin{pmatrix} a & b \\ c & d \end{pmatrix} \in \Mat(2,H_\mathbb{R}) \middle| ab^\ddagger = ba^\ddagger, \ cd^\ddagger = dc^\ddagger, \ ad^\ddagger - bc^\ddagger = 1\right\}.
	\end{align*}
	
\noindent One checks that this is a group, and that $SL^\ddagger(2,H_\mathbb{R})/\{\pm I\} \cong \text{ M\"{o}b}(\RR^3)$---or, if one prefers, since $\text{ M\"{o}b}(\RR^3) \cong SO^+(4,1)$, $SL^\ddagger(2,H_\mathbb{R}) \cong \text{Spin}(4,1)$, the universal cover of $SO^+(4,1)$. However, one observes that there is nothing in the definition of $SL^\ddagger(2,H_\mathbb{R})$ that is specific to the quaternions: one can just as well choose any ring equipped with an involution $\sigma$ and this will allow you to define a group $SL^\sigma(2,R)$. Our goal shall be describe what these groups are and under what circumstances they are isomorphic to one another. For example, if $R$ is an associative algebra over a field $F$ and $\sigma: R \rightarrow R$ is a morphism of affine $F$-varieties, then this an algebraic group. In fact, our first major result will be the following.

\begin{theorem}\label{dimension restriction}
Let $F$ be a field of characteristic not $2$, $A$ a central simple algebra over $F$ of dimension $n^2$, and $\sigma: A \rightarrow A$ an $F$-linear involution. Then $SL^\sigma(2,A)$ is a linear algebraic group over $F$. Specifically, it is either a symplectic group of dimension $2n^2 + n$, or an orthogonal group of dimension $2n^2 - n$.
\end{theorem}

\noindent Section \ref{SECTION CSA} proves this result along with more detailed information, such as working out the Lie algebra of this algebraic group. We primarily consider the case where $SL^\sigma(2,A)$ is symplectic since in that case it is simply-connected and we can get very detailed information about when such groups are isomorphic. In Section \ref{SECTION Orders and Arithmetic Groups}, we restrict to looking at algebraic number fields, as then we can consider orders $\OO$ of the central simple algebra $A$. If $\OO$ is closed under $\sigma$, then $SL^\sigma(2,\OO)$ is a well-defined group; in fact, it is an arithmetic subgroup of $SL^\sigma(2,\OO)$. Assuming $SL^\sigma(2,A)$ is symplectic, we can give a nice description of when such groups are isomorphic to one another in a sense that behaves well with respect to algebraic groups.

\begin{theorem}\label{general isomorphism theorem}
Let $K$ be an algebraic number field, $A$ a central simple algebra over $K$, $\sigma: A \rightarrow A$ an $F$-linear involution, and $\OO_1,\OO_2$ orders of $A$ closed under $\sigma$. If $SL^\sigma(2,A)$ is a symplectic group, then there exists a group isomorphism $\Psi: SL^\sigma(2,\OO_1) \rightarrow SL^\sigma(2,\OO_2)$ which lifts to an automorphism of $SL^\sigma(2,A)$ if and only if there exists a ring isomorphism $\Phi: \Mat(2,\OO_1) \rightarrow \Mat(2,\OO_2)$ such that $\Phi \circ \hat{\sigma} = \hat{\sigma} \circ \Phi$, where
    \begin{align*}
        \hat{\sigma}: \Mat(2,A) &\rightarrow \Mat(2,A) \\
        \begin{pmatrix} a & b \\ c & d \end{pmatrix} &\mapsto \begin{pmatrix} \sigma(d) & -\sigma(b) \\ -\sigma(c) & \sigma(a) \end{pmatrix}.
    \end{align*}
\end{theorem}

\noindent If we restrict to the case of rational quaternion algebras, then we can remove the requirement that the group isomorphism $\Psi: SL^\sigma(2,\OO_1) \rightarrow SL^\sigma(2,\OO_2)$ extends to an automorphism of $SL^\sigma(2,A)$, which is proved in Section \ref{SECTION Quaternion Orders}. In this setting, $SL^\sigma(2,H)$ is symplectic group if and only if $\sigma$ is $F$-linear and is not quaternion conjugation---such involutions are called orthogonal and we shall always denote such an involution by $\ddagger$. We then have the following.

\begin{theorem}\label{special isomorphism theorem}
Let $H_1, H_2$ be rational quaternion algebras with orthogonal involutions $\ddagger_1, \ddagger_2$. Let $\OO_1, \OO_2$ be orders of $H_1, H_2$ closed under $\ddagger_1, \ddagger_2$, respectively. Then $SL^{\ddagger_1}(2,\OO_1) \cong SL^{\ddagger_2}(2,\OO_2)$ as a group if and only if there exists a ring isomorphism $\Phi: \Mat(2,\OO_1) \rightarrow \Mat(2,\OO_2)$ such that $\Phi \circ \hat{\ddagger}_1 = \hat{\ddagger}_2 \circ \Phi$, where
    \begin{align*}
        \hat{\ddagger}_i: \Mat(2,H_i) &\rightarrow \Mat(2,H_i) \\
        \begin{pmatrix} a & b \\ c & d \end{pmatrix} &\mapsto \begin{pmatrix} d^{\ddagger_i} & -b^{\ddagger_i} \\ -c^{\ddagger_i} & a^{\ddagger_i} \end{pmatrix}.
    \end{align*}
\end{theorem}

\noindent This is interesting from a geometric point of view due to an accidental isomorphism between the symplectic groups $SL^\ddagger(2,H)$ quotiented by $\{\pm 1\}$ and orthogonal groups of indefinite, quinary quadratic forms, which we prove in Section \ref{SECTION QAs}.

\begin{theorem}\label{Correspondence between QAs and Orthogonal Groups}
Let $F$ be a characteristic $0$ field. Then there is a bijection
	\begin{align*}
	\left\{\substack{\text{Isomorphism classes of} \\ \text{quaternion algebras over } F}\right\} &\rightarrow \left\{\substack{\text{Isomorphism classes of} \\ \text{orthogonal groups of indefinite,} \\ \text{quinary quadratic forms over } F} \right\} \\
	[H] & \mapsto \left[SL^\ddagger(2,H)/\{\pm 1\}\right].
	\end{align*}
\end{theorem}

\noindent In particular, this implies that if $H$ is a rational, definite quaternion algebra, $\ddagger$ an orthogonal involution on $H$, and $\OO$ is a $\ddagger$-order, then $SL^\ddagger(2,\OO)/\{\pm 1\}$ is an arithmetic subgroup of $SO(4,1) \cong \Isom(\HH^4)$, the isometry group of hyperbolic $4$-space---this makes $SL^\ddagger(2,\OO)/\{\pm 1\}$ a Kleinian group, and so have an immediate corollary to Theorem \ref{special isomorphism theorem} as a consequence of Mostow rigidity.

\begin{corollary}
Let $H_1, H_2$ be rational quaternion algebras with orthogonal involutions $\ddagger_1, \ddagger_2$. Let $\OO_1, \OO_2$ be orders of $H_1, H_2$ closed under $\ddagger_1, \ddagger_2$, respectively. For $i = 1,2$, let $\Gamma_i = SL^{\ddagger_i}(2,\OO_i)/\{\pm 1\}$, which we think of as subgroups of $\text{Isom}(\HH^4)$. The following are equivalent.
    \begin{enumerate}
        \item The orbifolds $\HH^4/\Gamma_1,\HH^4/\Gamma_2$ are homotopic.
        \item The orbifolds $\HH^4/\Gamma_1,\HH^4/\Gamma_2$ are isometric.
        \item There exists a ring isomorphism $\Phi: \Mat(2,\OO_1) \rightarrow \Mat(2,\OO_2)$ such that $\Phi \circ \hat{\ddagger}_1 = \hat{\ddagger}_2 \circ \Phi$, where $\hat{\ddagger}_i$ is defined as in Theorem \ref{special isomorphism theorem}.
    \end{enumerate}
\end{corollary}

\noindent This property of the groups $SL^\ddagger(2,\OO)$ is analogous to the Bianchi groups $SL(2,\mathfrak{o}_K)$, where $K$ is an imaginary quadratic field and $\mathfrak{o}_K$ is its ring of integers. Recall that $PSL(2,\mathfrak{o}_K)$ can be viewed as an arithmetic subgroup of $\text{Isom}(\HH^3)$, and the corresponding orbifolds $\HH^3/PSL(2,\mathfrak{o}_K)$ are homotopic if and only if the rings of integers are isomorphic. The Bianchi groups have the additional property that they are in some sense maximal---$SL(2,\mathfrak{o}_K)$ is not contained inside any larger arithmetic subgroup of $SL(2,K)$. As it happens, our new groups also have this property, which shall be shown in Section \ref{SECTION Orders and Arithmetic Groups}.

\begin{theorem}\label{Maximal Arithmetic Groups for QAs}
Let $K$ be an algebraic number field, $H$ a quaternion algebra over $K$, $\ddagger$ an orthogonal involution on $H$, and $\OO$ an order of $H$ that is closed under $\ddagger$ but is not contained inside any larger order closed under $\ddagger$. Then $SL^\ddagger(2,\OO)$ is a maximal arithmetic subgroup of $SL^\ddagger(2,H)$ in the sense that it is not contained inside any larger arithmetic subgroup of $SL^\ddagger(2,H)$.
\end{theorem}

\noindent In Section \ref{SECTION Quaternion Orders}, we conclude with a variety of examples and counter-examples demonstrating that our results are in some sense sharp---for example, we show that in Theorem \ref{special isomorphism theorem}, one cannot replace $\Mat(2,\OO)$ with $\OO$ instead, which is a major discrepancy from the comparatively simpler theory for commutative rings.

\subsection*{Acknowledgements:}

The author would like to thank Joseph Quinn for a very productive conversation about invariants and isomorphisms of hyperbolic quotient manifolds, which inspired many of the approaches used in this paper, as well as Ara Basmajian and Abhijit Champanerkar, for asking pointed questions.

\section{General Rings:}\label{SECTION General Rings}

Let $R$ be a ring. An \emph{involution} on $R$ is a map $\sigma: R \rightarrow R$ such that for all $x,y \in R$,
    \begin{enumerate}
        \item $\sigma(\sigma(x)) = x$,
        \item $\sigma(x + y) = \sigma(x) + \sigma(y)$, and
        \item $\sigma(xy) = \sigma(y)\sigma(x)$.
    \end{enumerate}
    
\noindent Given two rings with involution $(R_1, \sigma_1)$, $(R_2, \sigma_2)$, a \emph{morphism} between them is a ring homomorphism $\varphi: R_1 \rightarrow R_2$ such that $\varphi(\sigma_1(x)) = \sigma_2(\varphi(x))$ for all $x \in R_1$. There are countless standard examples of involutions; for instance, if $F$ is a field, we can consider $M = \Mat(n,F)$, the ring of $n\times n$ matrices with coefficients in $F$, together with the adjugate map, usually written as $\dagger$. For our purposes, a particularly useful example is when $n = 2$---in that case, $R = \Mat(2,F)$ is a quaternion algebra over $F$, and the adjugate is
    \begin{align*}
        \begin{pmatrix} a & b \\ c & d \end{pmatrix}^\dagger = \begin{pmatrix} d & -b \\ -c & a \end{pmatrix},
    \end{align*}
    
\noindent which happens to be the standard involution, also known in this context as quaternion conjugation. More generally, for any commutative ring $S$, $R = \Mat(2,S)$ equipped with the adjugate is a ring with involution. The same is not true if the base ring is not commutative, but there certainly exists a way to fix this in the $n = 2$ case, as follows.

\begin{lemma}
Let $(R, \sigma)$ be a ring with involution. Then $\Mat(2,R)$ together with the map
    \begin{align*}
        \hat{\sigma}: \Mat(2,R) &\rightarrow \Mat(2,R) \\
        \begin{pmatrix} a & b \\ c & d \end{pmatrix} &\mapsto \begin{pmatrix} \sigma(d) & -\sigma(b) \\ -\sigma(c) & \sigma(a) \end{pmatrix}
    \end{align*}
    
\noindent is a ring with involution.
\end{lemma}

\begin{proof}
It is easy to see that $\hat{\sigma}$ squares to the identity and that it preserves addition. It remains to prove that it reverses multiplication. This is a straightforward computation. On the one hand,
    \begin{align*}
        \hat{\sigma}\left(\begin{pmatrix} a_1 & b_1 \\ c_1 & d_1 \end{pmatrix}\begin{pmatrix} a_2 & b_2 \\ c_2 & d_2 \end{pmatrix}\right) &= \hat{\sigma}\left(\begin{pmatrix} a_1a_2 + b_1 c_2 & a_1 b_2 + b_1 d_2 \\ c_1 a_2 + d_1 c_2 & c_1 b_2 + d_1 d_2 \end{pmatrix}\right) \\
        &= \begin{pmatrix} \sigma(c_1 b_2 + d_1 d_2) & -\sigma(a_1 b_2 + b_1 d_2) \\ -\sigma(c_1 a_2 + d_1 c_2) & \sigma(a_1 a_2 + b_1 c_2) \end{pmatrix}
    \end{align*}
    
\noindent On the other hand,
    \begin{align*}
        \hat{\sigma}\left(\begin{pmatrix} a_2 & b_2 \\ c_2 & d_2 \end{pmatrix}\right)&\hat{\sigma}\left(\begin{pmatrix} a_1 & b_1 \\ c_1 & d_1 \end{pmatrix}\right) \\ &= \begin{pmatrix} \sigma \left(d_2\right) & -\sigma \left(b_2\right) \\ -\sigma \left(c_2\right) & \sigma \left(a_2\right) \end{pmatrix}\begin{pmatrix} \sigma \left(d_1\right) & -\sigma \left(b_1\right) \\ -\sigma \left(c_1\right) & \sigma \left(a_1\right) \end{pmatrix} \\
        &= \begin{pmatrix} \sigma(d_2)\sigma(d_1) + \sigma(b_2)\sigma(c_1) & -\sigma(d_2)\sigma(b_1) - \sigma(b_2)\sigma(a_1) \\ -\sigma(c_2)\sigma(d_1) - \sigma(a_2)\sigma(c_1) & \sigma(c_2)\sigma(b_1) + \sigma(a_2)\sigma(a_1) \end{pmatrix}.
    \end{align*}
    
\noindent By the properties of $\sigma$, these are in fact the same.
\end{proof}

\begin{remark}
As an aside, it is entirely possible that there is a nice generalization of this argument for $n > 2$---however, it is not clear to the author how to do this, since in general the adjugate is defined in terms of the minors of the matrix, whereas there are many different generalizations of the determinant for non-commutative rings. This is left as a question for the reader to ponder.
\end{remark}

\begin{remark}
A slightly more standard way to define an involution on $\Mat(n,R)$ is by composing the transpose map with element-wise application of $\sigma: R \rightarrow R$---one checks that this works for any $n$. (Consult \cite{Involutions} for more details and examples.) For $n = 2$, this is related to $\hat{\sigma}$ in the following way:
    \begin{align*}
        \hat{\sigma}\left(\begin{pmatrix} a & b \\ c & d \end{pmatrix}\right) = \begin{pmatrix} 0 & 1 \\ -1 & 0 \end{pmatrix} \begin{pmatrix} \sigma(a) & \sigma(c) \\ \sigma(b) & \sigma(d) \end{pmatrix} \begin{pmatrix} 0 & 1 \\ -1 & 0 \end{pmatrix}^{-1}.
    \end{align*}
\end{remark}

In any case, this lemma allows us to make the following definition.

\begin{definition}\label{DEFINITION twisted SL}
Let $(R,\sigma)$ be a ring with involution. By the \emph{special linear group on $R^2$ twisted by $\sigma$}, we shall mean the group
    \begin{align*}
        SL^{\sigma}(2,R) := \left\{M \in \Mat(2,R)\middle|M\hat{\sigma}(M) = 1\right\}.
    \end{align*}
\end{definition}

It is easy to see that this really is a group under matrix multiplication. The name of this group is motivated as follows. For convenience, define
    \begin{align*}
        R^+ :&= \left\{x \in R\middle| \sigma(x) = x\right\} \\
        R^- :&= \left\{x \in R\middle| \sigma(x) = -x\right\}.
    \end{align*}
    
\noindent Then we have the following lemma.

\begin{lemma}
Let $(R,\sigma)$ be a ring with involution. Then
    \begin{align*}
        SL^{\sigma}(2,R) = \left\{\begin{pmatrix} a & b \\ c & d \end{pmatrix} \in \Mat(2,R)\middle|a\sigma(b), c\sigma(d) \in R^+, \ a\sigma(d) - b\sigma(c) = 1\right\}.
    \end{align*}
\end{lemma}

\begin{proof}
Choose any matrix
    \begin{align*}
        M=\begin{pmatrix} a & b \\ c & d \end{pmatrix} \in \Mat(2,R)
    \end{align*}
    
\noindent and compute
    \begin{align*}
        M\hat{\sigma}(M) &= \begin{pmatrix} a & b \\ c & d \end{pmatrix}\begin{pmatrix} \sigma(d) & -\sigma(b) \\ -\sigma(c) & \sigma(a) \end{pmatrix} \\
        &= \begin{pmatrix} a\sigma(d) - b\sigma(c) & -a\sigma(b) + b\sigma(a) \\ c\sigma(d) - d\sigma(c) & -c\sigma(b) + d\sigma(a)\end{pmatrix}.
    \end{align*}
    
\noindent This matrix is equal to $I$ if and only if
    \begin{align*}
        a\sigma(d) - b\sigma(c) &= 1 \\
        a\sigma(b) &= \sigma\left(a\sigma(b)\right) \\
        c\sigma(d) &= \sigma\left(c\sigma(d)\right),
    \end{align*}
    
\noindent which is to say that
    \begin{align*}
        M \in \left\{\begin{pmatrix} a & b \\ c & d \end{pmatrix} \in \Mat(2,R)\middle|a\sigma(b), c\sigma(d) \in R^+, \ a\sigma(d) - b\sigma(c) = 1\right\}.
    \end{align*}
    
\noindent Ergo, these two sets are in fact one and the same.
\end{proof}

It is now easy to see that Definition \ref{DEFINITION twisted SL} has two important special cases:
    \begin{enumerate}
        \item If $R$ is commutative, then the identity map $id: R \rightarrow R$ is an involution, and it is easy to see that
            \begin{align*}
                SL^{id}(2,R) = SL(2,R).
            \end{align*}
            
            \noindent That is, if the twist is trivial, then we simply reduce to the ordinary special linear group.
            
        \item If $R = H_\RR$, the Hamiltonian quaternions, and $\ddagger: H_\RR \rightarrow H_\RR$ is the map $(x + wi + yj + zk)^\ddagger = x + wi + yj - zk$, then $SL^\ddagger(2,H_\RR)$ is exactly the group we saw in the introduction, and so
            \begin{align*}
                SL^\ddagger(2,H_\RR)/\{\pm I\} \cong SO^+(3,1) \cong \Isom^0(\HH^4),
            \end{align*}
            
        \noindent the orientation-preserving isometry group of hyperbolic $4$-space.
    \end{enumerate}
    
\section{Central Simple Algebras:}\label{SECTION CSA}
    
Generalizing the example of $SL^\ddagger(2,H_\RR)$, one could take any field $F$, a central simple algebra $A$ over $F$, and an involution $\sigma: A \rightarrow A$. This group is also known under another name: it is $\text{Isom}(\Mat(2,A), \hat{\sigma})$, the collection of elements $M \in \Mat(2,A)$ with the property that $A\hat{\sigma}(A) = 1$, which are known as isometries. (Consult \cite{Involutions} for more details.) It is easily checked that for any central simple algebra $A$ over a field $F$, any involution $\sigma$ preserves its center; consequently, the restriction of $\sigma$ to $F$ is either the identity or an automorphism of order two. Involutions that fix $F$ are known as \emph{involutions of the first kind}; the other type are \emph{involutions of the second kind}. We will not consider involutions of the second kind here for a variety of reasons; ultimately, we will be interested in the case where $F = \QQ$, in which case all involutions are of the first kind. It is easy to see that for any $\lambda \in F$, $\hat{\sigma}(\lambda) = \sigma(\lambda)$, so $\hat{\sigma}$ is of the first kind if and only if $\sigma$ itself is.

Involutions of the first kind are further split into symplectic and orthogonal involutions. These can defined as follows: any involution of the first kind $\sigma: A \rightarrow A$ can be extended to an involution on $A \otimes_F \overline{F}$, where $\overline{F}$ is the algebraic closure of $F$. However, by the Artin-Wedderburn theorem, $A \otimes_F \overline{F} \cong \Mat(n, \overline{F})$ for some $n$. It readily checked that on such an algebra, any involution of the first kind $\sigma$ corresponds to a bilinear form $b_\sigma$ with the defining property that for all $v,w \in \overline{F}^n$ and all $M \in \Mat(n,\overline{F})$
    \begin{align*}
        b_\sigma(v,Mw) = b_\sigma(\sigma(M)v,w).
    \end{align*}
    
\noindent This form is unique up to multiplication by scalars, and it is either symmetric or alternating. If it is symmetric, we say that $\sigma$ is \emph{orthogonal}; it is alternating, we say that it is \emph{symplectic}.

\begin{lemma}\label{Orthogonal to Symplectic}
Let $F$ be a field, $A$ a central simple algebra over $F$, and $\sigma: A \rightarrow A$ be an involution of the first kind. If $\sigma$ is symplectic, then $\hat{\sigma}$ is orthogonal, and vice versa.
\end{lemma}

\begin{proof}
Let $b_\sigma$ be a bilinear form on $A \otimes_F \overline{F}$ such that for all $v,w \in \overline{F}^n$ and all $M \in \Mat(n,\overline{F})$
    \begin{align*}
        b_\sigma(v,Mw) = b_\sigma(\sigma(M)v,w).
    \end{align*}
    
\noindent We shall construct a corresponding bilinear form for $\hat{\sigma}$. Specifically, note that $\Mat(2,A \otimes_F \overline{F}) \cong \Mat(2n, \overline{F})$, and so we shall want a bilinear form $b_{\hat{\sigma}}$ on $\overline{F}^{2n} = \overline{F}^{n} \times \overline{F}^{n}$. We do this by
    \begin{align*}
        b_{\hat{\sigma}}\left((v_1, v_2), (w_1,w_2)\right) &= b_\sigma(v_1, w_2) - b_\sigma(v_2, w_1),
    \end{align*}
    
\noindent in which case for any
    \begin{align*}
        \gamma = \begin{pmatrix} M_1 & M_2 \\ M_3 & M_4 \end{pmatrix} \in \Mat(2,\Mat(n,\overline{F})) \cong \Mat(2n,\overline{F}),
    \end{align*}
    
\noindent we have
    \begin{align*}
        b_{\hat{\sigma}}\left((v_1, v_2), \gamma(w_1,w_2)\right) &= b_{\hat{\sigma}}\left((v_1, v_2), (M_1 w_1 + M_2 w_2, M_3 w_1 + M_4 w_2)\right) \\
        &= b_\sigma(v_1, M_3 w_1 + M_4 w_2) - b_\sigma(v_2, M_1 w_1 + M_2 w_2) \\
        &= b_\sigma(\sigma(M_3)v_1, w_1) + b_\sigma(\sigma(M_4) v_1, w_2) \\ &- b_\sigma(\sigma(M_1)v_2, w_1) - b_\sigma(\sigma(M_2)v_2, w_2) \\
        &= b_\sigma(\sigma(M_4)v_1 - \sigma(M_2)v_2, w_2) - b_\sigma(-\sigma(M_3)v_1 + \sigma(M_1)v_2, w_1) \\
        &= b_{\hat{\sigma}}\left((\sigma(M_4)v_1 - \sigma(M_2)v_2,-\sigma(M_3)v_1 + \sigma(M_1)v_2),(w_1,w_2)\right) \\
        &= b_{\hat{\sigma}}\left(\hat{\sigma}(\gamma)(v_1,v_2),(w_1,w_2)\right),
    \end{align*}
    
\noindent as desired. Note that if $\sigma$ is orthogonal, then
    \begin{align*}
        b_{\hat{\sigma}}\left((v_1, v_2), (v_1,v_2)\right) &= b_\sigma(v_1, v_2) - b_\sigma(v_2, v_1) = 0,
    \end{align*}
    
\noindent so the form is alternating and $\hat{\sigma}$ is symplectic. On the other hand, if $\sigma$ is symplectic, then
    \begin{align*}
        b_{\hat{\sigma}}\left((v_1, v_2), (w_1,w_2)\right) &- b_{\hat{\sigma}}\left((w_1, w_2), (v_1,v_2)\right) \\
        &= b_\sigma(v_1, w_2) - b_\sigma(v_2, w_1) - b_\sigma(w_1, v_2) + b_\sigma(w_2, v_1) \\
        &= b_\sigma(v_1 + w_2, v_1 + w_2) - b_\sigma(v_2 + w_1, v_2 + w_1) = 0,
    \end{align*}
    
\noindent so the form is symmetric and $\hat{\sigma}$ is orthogonal.
\end{proof}

\begin{remark}
Notice that because did not identify alternating and skew-symmetric forms, this proof applies perfectly well in characteristic $2$, where those two types of forms are distinct.
\end{remark}

\begin{corollary}\label{Is symplectic or orthogonal}
Let $F$ be a field, $A$ a central simple algebra over $F$, and $\sigma: A \rightarrow A$ be an involution of the first kind. If $\sigma$ is orthogonal, then $SL^\sigma(2,A)$ is a symplectic group. Otherwise, it is an orthogonal group.
\end{corollary}

\begin{proof}
As was shown in the proof of Lemma \ref{Orthogonal to Symplectic}, $SL^\sigma(2,A \otimes_F \overline{F})$ is the group of linear transformations that preserve a bilinear form $b_{\hat{\sigma}}$ which is alternating if $\sigma$ is orthogonal, and symmetric otherwise. In the first case, we have that $SL^\sigma(2,A \otimes_F \overline{F}) = \text{Sp}(b_{\hat{\sigma}})$, the symplectic group of $b_{\hat{\sigma}}$; in the second case, we have that $SL^\sigma(2,A \otimes_F \overline{F}) = O(b_{\hat{\sigma}})$, the orthogonal group of $b_{\hat{\sigma}}$.
\end{proof}

Since $SL^\sigma(2,A)$ is an algebraic group, we can work out its Lie algebra. In this context, a Lie algebra is a subspace of an associative $F$-algebra that is closed under the bracket $[x,y] = xy - yx$. Any algebraic group $G$ has an associated Lie algebra consisting of its left invariant derivations---that is, if $R = F[G]$ is the space of regular functions on $G$, then it is the subspace
    \begin{align*}
        \mathcal{L}(G) &= \left\{\delta \in \text{Der}(R)\middle|\delta \circ \lambda_x = \lambda_x \circ \delta, \ \forall x \in G\right\}
    \end{align*}

\noindent where $\text{Der}(R)$ is the set of derivations on $R$ and $(\lambda_x f)(y) = f(x^{-1}y)$ for any $f \in R$. Of course, the Lie algebras of the symplectic and orthogonal groups are known, but here we can give a nice description in terms of the central simple algebra $A$.

\begin{definition}
Let $F$ be a field, $A$ a central simple algebra over $F$, $\sigma: A \rightarrow A$ an involution of the first kind. Then we define
    \begin{align*}
        \mathfrak{sl}^\sigma(2,A) = \left\{X \in \Mat(2,A)\middle| \hat{\sigma}(X) = -X\right\}.
    \end{align*}
\end{definition}

\begin{lemma}
$\mathfrak{sl}^\sigma(2,A)$ is a Lie algebra.
\end{lemma}

\begin{proof}
Since $\sigma$ is $F$-linear, $\hat{\sigma}$ is $F$-linear. Therefore $\mathfrak{sl}^\sigma(2,A)$ is an $F$-vector space. Checking that it is closed under the bracket is also easy, since
    \begin{align*}
        \hat{\sigma}(XY - YX) &= \hat{\sigma}(Y)\hat{\sigma}(X) - \hat{\sigma}(X)\hat{\sigma}(Y) \\
        &= -(XY - YX),
    \end{align*}
    
\noindent concluding the proof.
\end{proof}

\begin{theorem}\label{Lie algebra computation}
Let $F$ be a field, $A$ a central simple algebra over $F$, $\sigma: A \rightarrow A$ an involution of the first kind. Then $\mathfrak{sl}^\sigma(2,A)$ is isomorphic to the Lie algebra of $SL^\sigma(2,A)$.
\end{theorem}

\begin{proof}
The group $SL^\sigma(2,A)$ is a closed subgroup of $GL(2,A)$. The Lie algebra of $GL(2,A)$, $\mathfrak{gl}(2,A)$, is isomorphic to $\Mat(2,A)$ via the map
    \begin{align*}
        \Mat(2,A) &\rightarrow \mathfrak{gl}(2,A) \\
        M &\mapsto \left(f \mapsto \left(X \mapsto \left.\frac{d}{dt}f(X(I + tM))\right|_{t = 0}\right)\right).
    \end{align*}
    
\noindent However, it is more convenient to use the fact that the Lie algebra can be identified with the tangent space at $I$---this we can define as the ideal of rational functions $p: F \rightarrow GL(2,A)$ such that $p(0) = I$, quotiented by the square of this ideal. In that case, the Lie algebra of $SL^\sigma(2,A)$ can be viewed as the sub-ideal consisting of all rational functions $p: F \rightarrow SL^\sigma(2,A)$. By our isomorphism, this is equivalent to finding all $M \in \Mat(2,A)$ such that
    \begin{align*}
        (I + tM)\hat{\sigma}(I + tM) = I + O(t^2).
    \end{align*}
    
\noindent However,
    \begin{align*}
        (I + tM)\hat{\sigma}(I + tM) &= (I + tM)(I + t\hat{\sigma}(M)) \\
        &= I + tM + t\hat{\sigma}(M) + t^2 M \hat{\sigma}(M),
    \end{align*}
    
\noindent so this actually happens if and only if $\hat{\sigma}(M) = -M$.
\end{proof}

One benefit of having this explicit description of the Lie algebra is that it makes it very easy to see what the dimension of $SL^\sigma(2,A)$ is.

\begin{corollary}\label{dimension computation}
Let $F$ be a field, $A$ a central simple algebra over $F$, $\sigma: A \rightarrow A$ an involution of the first kind. Then $SL^\sigma(2,A)$ is a linear algebraic group over $F$ of dimension
    \begin{align*}
        \dim_F(A) + 2\dim_F(A^+).
    \end{align*}
\end{corollary}

\begin{proof}
The dimension of $SL^\sigma(2,A)$ is the dimension of $\mathfrak{sl}^\sigma(2,A)$. But
    \begin{align*}
        \begin{pmatrix} a & b \\ c & d \end{pmatrix} = -\hat{\sigma}\left(\begin{pmatrix} a & b \\ c & d \end{pmatrix}\right) = \begin{pmatrix} -\sigma(a) & \sigma(b) \\ \sigma(c) & -\sigma(d) \end{pmatrix}
    \end{align*}
    
\noindent if and only if $d = -\sigma(a)$ and $b,c \in A^+$.
\end{proof}

It is known that if the characteristic is not $2$, then
    \begin{align*}
        \dim_F(A^+) = \begin{cases} \frac{n(n + 1)}{2} & \text{if $n$ is orthogonal} \\ \frac{n(n - 1)}{2} & \text{if $n$ is symplectic}, \end{cases}
    \end{align*}
    
\noindent (see \cite{Involutions}), and so Theorem \ref{dimension restriction} is an immediate corollary of this and Corollary \ref{Is symplectic or orthogonal}.

\section{Orders and Arithmetic Groups:}\label{SECTION Orders and Arithmetic Groups}

If $A$ is a central simple algebra over an algebraic number field $K$, then we can consider orders of $A$---in this context, an \emph{order} is a sub-ring $\OO \subset A$ that is finitely-generated as a $\mathfrak{o}_K$-module, where $\mathfrak{o}_K$ is the ring of integers of $K$, and such that $K\OO = A$. If $\sigma: A \rightarrow A$ is an involution and $\sigma(\OO) = \OO$, then we say that $\OO$ is a $\sigma$-\emph{order} of $A$. If it is not contained inside any larger $\sigma$-order, then we say that it is a \emph{maximal $\sigma$-order}. Such orders were studied by Scharlau \cite{Scharlau1974} in the 1970s and then generalized to Azumaya algebras by Saltmann \cite{Saltman1978}. More recently, the author demonstrated how to classify the $\sigma$-orders of a quaternion algebra over local and global fields, which we will discuss in Section \ref{SECTION Quaternion Orders}. In the meantime, we consider some of the broader theory.

If $\sigma$ is an involution of the first kind on $A$ and $\OO$ is a $\sigma$-order of $A$, then it is easy to see that $SL^\sigma(2,\OO)$ is an arithmetic subgroup of $SL^\sigma(2,A)$. If $\sigma$ is an orthogonal involution, then we can say significantly more; this is because $SL^\sigma(2,A)$ is a symplectic group, and it is therefore an almost simple, simply-connected group. Moreover, it is easy to see that if $\nu$ is an infinite place of $\mathfrak{o}_K$ then $SL^\sigma(2,A_\nu)$ is not a compact group since none of the symplectic groups over $\RR$ or $\CC$ are; here $A_\nu = A \otimes_K K_\nu$, where $K_\nu$ is the localization of $K$ at the place $\nu$. Therefore, we can use the strong approximation theorem proved by Knesser and Platonov \cite{Kneser1965, Platonov1969}.

\begin{theorem}\label{Strong Approximation}
Let $K$ be an algebraic number field, $A$ a central simple algebra over $K$, $\sigma$ an orthogonal involution on $A$, and $S$ be the set of infinite places of $K$. Then $SL^\sigma(2,A)$ has strong approximation with respect to $S$.
\end{theorem}

The version of strong approximation that we will need can be summarized as follows.

\begin{corollary}\label{Strong Approximation Useful}
Let $K$ be an algebraic number field, $A$ a central simple algebra over $K$, $\sigma$ an orthogonal involution on $A$, $\OO$ a $\sigma$-order of $A$, and $S$ be the set of infinite places of $K$. Then $SL^\sigma(2,\OO)$ is dense inside
    \begin{align*}
        \prod_{\mathfrak{p} \text{ prime}}SL^\sigma(2,\OO_\mathfrak{p})
    \end{align*}
    
\noindent given the direct product topology, where $\OO_\mathfrak{p} = \OO \otimes_{\mathfrak{o}_K} \mathfrak{o}_{K, \nu}$, $\mathfrak{o}_K$ is the ring of integers of $K$, and $\mathfrak{o}_{K, \nu}$ is the ring of integers of $K_\nu$.
\end{corollary}

The reason why this is helpful is that it will allow us to compute an important invariant, the matrix ring of the group.

\begin{definition}
Let $K$ be an algebraic number field, $A$ a central simple algebra over $K$, and $\sigma$ an orthogonal involution on $A$. Let $\Gamma \subset SL^\sigma(2,A)$ be a subgroup. The \emph{matrix ring} $\mathfrak{o}_K[\Gamma]$ of $\Gamma$ is the subring of $\Mat(2,A)$ generated by $\mathfrak{o}_K$ and the elements of $\Gamma$.
\end{definition}

\begin{remark}
One might ask in what sense this is an invariant. If $\Gamma, \Gamma'$ are conjugate to one another in $GL(2,A \otimes_K \overline{K})$, then it is clear that this extends to an isomorphism
    \begin{align*}
        \mathfrak{o}_K[\Gamma] &\rightarrow \mathfrak{o}_K[\Gamma'] \\
        M &\mapsto \gamma M \gamma^{-1}
    \end{align*}
    
\noindent for some $\gamma \in GL(2,H \otimes_K \overline{K})$. Thus, in some sense any isomorphism of arithmetic groups leaves the matrix rings of the groups invariant. We shall see later that in special cases we can generalize this to group isomorphisms.
\end{remark}

\begin{lemma}\label{Group Ring}
Let $K$ be an algebraic number field, $A$ a central simple algebra over $K$, $\sigma$ an orthogonal involution on $A$, and $\OO$ a $\sigma$-order of $A$. Then $\mathfrak{o}_K\left[SL^\sigma(2,\OO)\right] = \Mat(2,\OO)$.
\end{lemma}

\begin{proof}
Obviously, $\mathfrak{o}_K\left[SL^\sigma(2,\OO)\right]$ is contained inside of $\Mat(2,\OO)$, so it shall suffice to show that every element of $\Mat(2,\OO)$ is contained in $\mathfrak{o}_K\left[SL^\sigma(2,\OO)\right]$. Since $\mathfrak{o}_K \subset \OO$, we know that $\Mat(2,\mathfrak{o}_K) \subset \mathfrak{o}_K\left[SL^\sigma(2,\OO)\right]$. Now, let $\hat{\mathfrak{o}}_K$ be the inverse limit
    \begin{align*}
        \hat{\mathfrak{o}}_K = \varprojlim_{\mathfrak{a} \text{ ideal}} \mathfrak{o}_K/\mathfrak{a} = \prod_{\mathfrak{p} \text{ prime}} \mathfrak{o}_{K, \mathfrak{p}}.
    \end{align*}
    
\noindent This is the completion of $\mathfrak{o}_K$ given the structure of a topological ring by specifying that the prime ideals $\mathfrak{p}$ form a basis of open sets for $0$. Define
    \begin{align*}
        \hat{\OO} = \prod_{\mathfrak{p} \text{ prime}} \OO_\mathfrak{p} = \prod_{\mathfrak{p} \text{ prime}} \OO \otimes_{\mathfrak{o}_K}\mathfrak{o}_{K, \mathfrak{p}} = \OO \otimes_{\mathfrak{o}_K} \hat{\mathfrak{o}}_K,
    \end{align*}
    
\noindent giving it the structure of a topological ring via the product topology. Thinking of $\hat{\OO}$ as a group, we can take the quotient $\hat{\OO}/\hat{\mathfrak{o}}_K$, which inherits the quotient topology. All of this allows us to define the following continuous projection.
    \begin{align*}
        \Psi: \prod_{\mathfrak{p} \text{ prime}} SL^\ddagger(2,\OO_\mathfrak{p}) &\rightarrow \hat{\OO}/\hat{\mathfrak{o}}_K \\
        \begin{pmatrix} a & b \\ c & d \end{pmatrix} &\mapsto a + \hat{\mathfrak{o}}_K.
    \end{align*}

\noindent This projection is surjective---this is because for any $z_\mathfrak{p} \in \OO_\mathfrak{p}$, there exists some $\lambda_\mathfrak{p} \in \mathfrak{o}_{K,\mathfrak{p}}$ such that $\lambda_\mathfrak{p} + z_\mathfrak{p} \in \OO_\mathfrak{p}^\times$. So, defining $\lambda \in \hat{\mathfrak{o}}_K$, $z \in \hat{\OO}$ such that their $\mathfrak{p}$-th coordinates are $\lambda_\mathfrak{p}$ and $z_\mathfrak{p}$ respectively, we have
	\begin{align*}
	\underbrace{\begin{pmatrix} \lambda + z & 0 \\ 0 & \left(\lambda + \sigma(z)\right)^{-1} \end{pmatrix}}_{\in SL^\sigma(2,\hat{\OO})} \mapsto z + \hat{\mathfrak{o}}_K.
	\end{align*}
	
\noindent However, by Corollary \ref{Strong Approximation Useful}, we know that $SL^\sigma(2,\OO)$ is dense, hence its image is also dense, which is to say that for any $a \in \OO/\mathfrak{o}_K$ and any prime ideal $\mathfrak{p}$, there exists $\gamma \in SL^\sigma(2,\OO)$ such that $\Psi(\gamma) - a \in \mathfrak{p}\OO$. Choose a basis $\{e_i\}_{i = 1}^n$ for $\OO$ and an ideal $\mathfrak{a} \subset \mathfrak{o}_K$ such that $\{e_i + \alpha_i\}_{i = 1}^n$ is a basis for all $\alpha_i \in \mathfrak{a}\OO$. Then there exist $\lambda_i \in \mathfrak{o}_K$ and $\alpha_i \in \mathfrak{a}\OO$ such that there exists $\gamma_i \in SL^\sigma(2,\OO)$ such $\Psi(\gamma_i) = \lambda_i + e_i + \alpha_i$, and therefore
    \begin{align*}
    \begin{pmatrix} 1 & 0 \\ 0 & 0 \end{pmatrix} \gamma_i \begin{pmatrix} 1 & 0 \\ 0 & 0 \end{pmatrix} - \begin{pmatrix} \lambda_i & 0 \\ 0 & 0 \end{pmatrix} = \begin{pmatrix} e_i + \alpha_i & 0 \\ 0 & 0 \end{pmatrix} \in \mathfrak{o}_K\left[SL^\sigma(2,\OO)\right].
    \end{align*}
    
\noindent However, since $\{e_i + \alpha_i\}_{i = 1}^n$ is a basis, it follows that
    \begin{align*}
        \begin{pmatrix} \OO & 0 \\ 0 & 0 \end{pmatrix} \subset \mathfrak{o}_K\left[SL^\sigma(2,\OO)\right].
    \end{align*}
    
\noindent By conjugating with elements in $\Mat(2,\mathfrak{o}_K)$, we get that $\Mat(2,\OO)\subset \mathfrak{o}_K\left[SL^\sigma(2,\OO)\right]$, as desired.
\end{proof}

More generally, for other arithmetic subgroups of $SL^\sigma(2,A)$, the matrix ring is an order of $\Mat(2,A)$.

\begin{lemma}\label{Lattices beget lattices}
Let $K$ be an algebraic number field, $A$ a central simple algebra over $K$, $\sigma$ an orthogonal involution on $A$, and $\Gamma$ an arithmetic subgroup of $SL^\sigma(2,A)$. Then $\mathfrak{o}_K\left[\Gamma\right]$ is an order of the central simple algebra $\Mat(2,A)$.
\end{lemma}

\begin{proof}
First, note that since $\Gamma \subset SL^\sigma(2,A)$, $\mathfrak{o}_K\left[\Gamma\right] \subset \mathfrak{o}_K\left[SL^\sigma(2,A)\right] = \Mat(2,A)$. Since $\Gamma$ is an arithmetic group, for some integer $l$, there is a morphism $\Psi:\Gamma \rightarrow SL(l,\mathfrak{o}_K)$ with a finite kernel. It is easy to see that $\mathfrak{o}_K\left[SL(l,\mathfrak{o}_K)\right] = \Mat(l,\mathfrak{o}_K)$ is a finitely-generated, Noetherian $\mathfrak{o}_K$-module. Therefore, the sub-module $\mathfrak{o}_K\left[\Psi(\Gamma)\right]$ is finitely-generated. This, in turn, means that $\mathfrak{o}_K\left[\Gamma\right]$ is finitely-generated as an $\mathfrak{o}_K$-module. Since it is a subring of the finite-dimensional algebra $\Mat(2,A)$, it is therefore an order.
\end{proof}

We can use this invariant to characterize when arithmetic subgroups are conjugate to each other.

\begin{theorem}\label{Conjugation turns into isomorphism}
Let $K$ be an algebraic number field, $A$ a central simple algebra over $K$, $\sigma$ an orthogonal involution on $A$. Let $\Gamma_1, \Gamma_2$ be arithmetic subgroups of $SL^\sigma(2,A)$ such that $\Gamma_i = \mathfrak{o}_K[\Gamma_i] \cap SL^\sigma(2,A)$. The following are equivalent.
    \begin{enumerate}
        \item There exists an element $\gamma \in SL^\sigma(2,A\otimes_K \overline{K})$ such that $\Gamma_2 = \gamma \Gamma_1 \gamma^{-1}$.
        \item There exists a group isomorphism $\Psi: \Gamma_1 \rightarrow \Gamma_2$ that extends to an automorphism of $SL^\sigma(2,A)$ as an algebraic group.
        \item $(\mathfrak{o}_K[\Gamma_1],\hat{\sigma}) \cong (\mathfrak{o}_K[\Gamma_2],\hat{\sigma})$.
    \end{enumerate}
\end{theorem}

\begin{proof}
Since $\Gamma_1, \Gamma_2$ are arithmetic groups, if $\Gamma_2 = \gamma \Gamma_1 \gamma^{-1}$ then this extends to an automorphism of $SL^\sigma(2,A)$; thus, the first condition certainly implies the second. On the other hand, any automorphism of $SL^\sigma(2,A)$ must come from an automorphism of the Lie algebra $\mathfrak{sl}^\sigma(2,A)$, since the group is simply-connected. Since the Lie algebra is of type $C_n$, all of its automorphisms are inner, which is to say that they arise as conjugation by elements in $SL^\sigma(2,A\otimes_K \overline{K})$---therefore, the second condition implies the first. Now, if $\Gamma_2 = \gamma \Gamma_1 \gamma^{-1}$, then this extends to a ring isomorphism
    \begin{align*}
        \mathfrak{o}_K[\Gamma_1] &\rightarrow \mathfrak{o}_K[\Gamma_1] \\
        M &\mapsto \gamma M \gamma^{-1}.
    \end{align*}
    
\noindent If $\gamma \in SL^\sigma(2,A \otimes_K \overline{K})$, then we see that
    \begin{align*}
        \sigma\left(\gamma M \gamma^{-1}\right) &= \sigma(\gamma^{-1})\sigma(M)\sigma(\gamma) \\
        &= \gamma \sigma(M) \gamma^{-1},
    \end{align*}
    
\noindent hence this is actually an isomorphism of rings with involution. On the other hand, since $\Gamma_1, \Gamma_2$ are arithmetic, any isomorphism $(\mathfrak{o}_K[\Gamma_1],\hat{\sigma}) \cong (\mathfrak{o}_K[\Gamma_2],\hat{\sigma})$ extends to an automorphism of $(\Mat(2,A),\hat{\sigma})$. Since $\Mat(2,A)$ is a central simple algebra, there exists $\gamma \in GL(2,A)$ such that this automorphism has the form $M \mapsto \gamma M \gamma^{-1}$. Since this is an automorphism of rings with involution, it must be that
    \begin{align*}
        \sigma(\gamma M \gamma^{-1}) = \sigma(\gamma^{-1}) \sigma(M) \sigma(\gamma) = \gamma\sigma(M)\gamma^{-1}
    \end{align*}
    
\noindent for all $M \in \Mat(2,A)$, which implies that $\gamma\sigma(\gamma)$ is in the center of $\Mat(2,A)$, which is $K$. Therefore, there exists $\lambda \in \overline{K}$ such that $\gamma' = \lambda \gamma$ satisfies the property $\gamma'\sigma(\gamma') = 1$, which is to say that $\gamma' \in SL^\sigma(2,A \otimes_K \overline{K})$. Therefore, we can take our desired isomorphism to be of the form
    \begin{align*}
        \mathfrak{o}_K[\Gamma_1] &\rightarrow \mathfrak{o}_K[\Gamma_1] \\
        M &\mapsto \gamma' M {\gamma'}^{-1}.
    \end{align*}
    
\noindent Since $\mathfrak{o}_K[\Gamma_i] \cap SL^\sigma(2,A) = \Gamma_i$, this isomorphism restricts to a group isomorphism from $\Gamma_1$ to $\Gamma_2$ by looking at the subgroup of elements $M$ such that $M\hat{\sigma}(M) = 1$.
\end{proof}

\begin{remark}
Theorem \ref{general isomorphism theorem} is an immediate corollary of this result and Lemma \ref{Group Ring}.
\end{remark}

We shall see later in the context of quaternion algebras, that while $(\OO_1, \sigma) \cong (\OO_2, \sigma)$ certainly implies $(\Mat(2,\OO_1),\hat{\sigma}) \cong (\Mat(2,\OO_2), \hat{\sigma})$, the converse is false; in fact, there are examples where $\OO_1 \ncong \OO_2$, but $(\Mat(2,\OO_1),\hat{\sigma}) \cong (\Mat(2,\OO_2), \hat{\sigma})$. Requiring $(\OO_1, \sigma) \cong (\OO_2, \sigma)$ actually corresponds to a significantly more stringent notion of isomorphism.

\begin{theorem}\label{Special Isomorphism of Orders Is Equivalent to Special Isomorphism of Groups}
Let $K$ be an algebraic number field, $A$ a central simple algebra over $K$, $\sigma$ an orthogonal involution on $A$, and $\OO_1, \OO_2$ be $\sigma$-orders of $A$. The following are equivalent.
	\begin{enumerate}
	\item There exists $\gamma \in SL^\sigma(2,A \otimes_K \overline{K})$ such that $\gamma SL^\sigma(2,\OO_1) \gamma^{-1} = SL^\sigma(2,\OO_2)$ and $\gamma M \gamma^{-1} = M$ for all $M \in SL(2,K)$.
	\item $(\OO_1,\ddagger) \cong (\OO_2,\ddagger)$.
	\end{enumerate}
\end{theorem}

\begin{proof}
If $(\OO_1, \sigma) \cong (\OO_2,\sigma)$, then this isomorphism extends to an automorphism of $(A,\sigma)$. Since $A$ is a central simple algebra, this automorphism must take the form $x \mapsto uxu^{-1}$ for some $u \in A^\times$. Furthermore, it must be that $u\sigma(u) \in K^\times$. Therefore, there exists some $\lambda \in \overline{K}$ such that $v = \lambda u$ satisfies $v\sigma(v) = 1$ and
    \begin{align*}
        \OO_1 & \rightarrow \OO_2 \\
        x &\mapsto vxv^{-1}
    \end{align*}
    
\noindent is the desired isomorphism of rings with involution. Now, we have that
    \begin{align*}
        \gamma = \begin{pmatrix} v & 0 \\ 0 & \sigma(v)^{-1} \end{pmatrix} \in SL^\sigma(2,A \otimes_K \overline{K})
    \end{align*}
    
\noindent and therefore
    \begin{align*}
        SL^\sigma(2,A) &\mapsto SL^\sigma(2,A) \\
        M &\mapsto \gamma M \gamma^{-1}
    \end{align*}
    
\noindent is an automorphism. However,
    \begin{align*}
        \gamma \begin{pmatrix} a & b \\ c & d \end{pmatrix} \gamma^{-1} = \begin{pmatrix} vav^{-1} & vbv^{-1} \\ vcv^{-1} & vdv^{-1} \end{pmatrix},
    \end{align*}
    
\noindent and therefore $\gamma SL^\sigma(2,\OO_1) \gamma^{-1} = SL^\sigma(2,\OO_2)$ and this map acts as the identity on $SL(2,K)$. In the other direction, suppose that there exists an element $\gamma \in SL^\sigma(2,H \otimes_K \overline{K})$ such that $\gamma SL^\sigma(2,\OO_1) \gamma^{-1} = SL^\sigma(2,\OO_2)$ and $\gamma M \gamma^{-1} = M$ for all $M \in SL(2,K)$. Then, by Lemma \ref{Group Ring}, the conjugation map extends to a ring isomorphism
	\begin{align*}
	\Mat(2,\OO_1) &\rightarrow \Mat(2,\OO_2) \\
	M &\mapsto \gamma M \gamma^{-1}.
	\end{align*}
	
\noindent This map restricts to the identity on $\Mat(2,\mathfrak{o}_K)$. This means that if we define subrings
	\begin{align*}
	U_i = \left\{M \in \Mat(2,\OO_i) \middle| M\begin{pmatrix} 1 & 1 \\ 0 & 1 \end{pmatrix} = \begin{pmatrix} 1 & 1 \\ 0 & 1 \end{pmatrix}M\right\},
	\end{align*}
	
\noindent then we are guaranteed that the ring isomorphism between the $\Mat(2,\OO_i)$ restricts to a ring isomorphism between the $U_i$. However, it is easy to see that $M \in U_i$ if and only if it is of the form
    \begin{align*}
        \begin{pmatrix} s & t \\ 0 & s \end{pmatrix}
    \end{align*}
    
\noindent for some $s, t \in \OO_i$. Therefore, for any $z \in \OO_1$,
	\begin{align*}
	\gamma \begin{pmatrix} 1 & z \\ 0 & 1 \end{pmatrix} \gamma^{-1} \in U_2.
	\end{align*}
	
\noindent Write
	\begin{align*}
	\gamma = \begin{pmatrix} a & b \\ c & d \end{pmatrix} \in SL^\sigma(2,A \otimes_K \overline{K}),
	\end{align*}
	
\noindent and note that
	\begin{align*}
	\gamma \begin{pmatrix} 0 & z \\ 0 & 0 \end{pmatrix} \gamma^{-1} = \begin{pmatrix} * & * \\ cz\sigma(c) & * \end{pmatrix},
	\end{align*}
	
\noindent hence $cz\sigma(c) = 0$. Since $cz\sigma(c) = 0$ for all $z \in \OO_1$, it must be that $cz\sigma(c) = 0$ for all $z \in A \otimes_K \overline{K}$. Since $A \otimes_K \overline{K} \cong \Mat(n,\overline{K})$ for some $n$, we can think of $c,z$ as linear transformations. Suppose that there exists $v \in \overline{K}^n$ such that $cv \neq 0$. Then there must also exist $w \in \overline{K}^n$ such that $\sigma(c)w \neq 0$. Therefore, there exists $z \in \Mat(n, \overline{K})$ such that $z\sigma(c)w = v$, which means that $cz\sigma(v)w \neq 0$. This is contradicted by the fact that $cz\sigma(c) = 0$ identically, which means that $c = 0$. By a similar argument with the sub-ring
	\begin{align*}
	L_i = \left\{M \in \Mat(2,\OO_i) \middle| M\begin{pmatrix} 1 & 0 \\ 1 & 1 \end{pmatrix} = \begin{pmatrix} 1 & 0 \\ 1 & 1 \end{pmatrix}M\right\},
	\end{align*}
	
\noindent we can prove that $b = 0$ as well. Thus $\gamma$ is a diagonal matrix, which is to say that
	\begin{align*}
	\gamma = \begin{pmatrix} u & 0 \\ 0 & \sigma(u)^{-1} \end{pmatrix}
	\end{align*}
	
\noindent for some $u \in \left(A \otimes_K \overline{K}\right)^\times$. Thus $\OO_2 = u\OO_1 u^{-1}$. Since
	\begin{align*}
	\begin{pmatrix} u & 0 \\ 0 & \sigma(u)^{-1} \end{pmatrix} \begin{pmatrix} 1 & 1 \\ 0 & 1 \end{pmatrix} \begin{pmatrix} u^{-1} & 0 \\ 0 & \sigma(u) \end{pmatrix} &= \begin{pmatrix} 1 & u\sigma(u) \\ 0 & 1 \end{pmatrix} = \begin{pmatrix} 1 & 1 \\ 0 & 1 \end{pmatrix},
	\end{align*}
	
\noindent we have that $u\sigma(u) = 1$, which means that the map
    \begin{align*}
        \Mat(2,\OO_1) &\rightarrow \Mat(2,\OO_2) \\
        M &\mapsto uMu^{-1}
    \end{align*}
    
\noindent is an isomorphism of rings with involution.
\end{proof}

We end this section by proving that if $\OO$ happens to be a maximal $\sigma$-order, then $SL^\sigma(2,\OO)$ is also a maximal arithmetic group in some sense.

\begin{theorem}\label{Maximal Arithmetic Groups}
Let $K$ be an algebraic number field, $A$ a central simple algebra over $K$, $\sigma$ an orthogonal involution on $A$, and $\OO$ a maximal $\sigma$-order of $A$. Then $SL^\sigma(2,\OO)$ is a maximal arithmetic subgroup of $SL^\sigma(2,A)$ in the sense that it is not contained inside any larger arithmetic subgroup of $SL^\sigma(2,A)$.
\end{theorem}

\begin{proof}
Suppose that $\Gamma$ is an arithmetic group containing $SL^\sigma(2,\OO)$. By Lemma \ref{Lattices beget lattices}, we know that $\mathfrak{o}_K\left[\Gamma\right]$ is an order of $\Mat(2,A)$ which, by Lemma \ref{Group Ring} contains $\Mat(2,\OO)$. Choose any element $\gamma \in \Gamma$, and choose any one of its coordinates $x$. Since the matrix ring contains
	\begin{align*}
	\begin{pmatrix} 1 & 0 \\ 0 & 0 \end{pmatrix},\begin{pmatrix} 0 & 1 \\ 0 & 0 \end{pmatrix},\begin{pmatrix} 0 & 0 \\ 1 & 0 \end{pmatrix},\begin{pmatrix} 0 & 0 \\ 0 & 1 \end{pmatrix},
	\end{align*}
	
\noindent it is clear that $\mathfrak{o}_K\left[\Gamma\right]$ must contain $\Mat(2,\OO[x])$. However, since
	\begin{align*}
	\begin{pmatrix} a & b \\ c & d \end{pmatrix}^{-1} = \begin{pmatrix} \sigma(d) & -\sigma(b) \\ -\sigma(c) & \sigma(a) \end{pmatrix},
	\end{align*}
	
\noindent we see that it must actually contain $\Mat(2,\OO[x,\sigma(x)])$. Since $\Mat(2,\OO[x,\sigma(x)])$ is a subring of the matrix ring, it must also be an order, from which we get that $\OO[x,\sigma(x)]$ is an order. However, $\OO[x,\sigma(x)]$ is clearly closed under $\sigma$, hence it is a $\sigma$-order. Since $\OO$ is a maximal $\sigma$-order, it follows that $\OO[x,\sigma(x)] = \OO$. Therefore, $\Gamma \subset \Mat(2,\OO)$. However, $SL^\sigma(2,\OO) = \Mat(2,\OO) \cap SL^\sigma(2,A)$, therefore $\Gamma = SL^\sigma(2,\OO)$.
\end{proof}

\begin{remark}
Theorem \ref{Maximal Arithmetic Groups for QAs} is nothing more than a special case of this result.
\end{remark}

\section{Quaternion Algebras:}\label{SECTION QAs}

We shall now explore the special case where $A = H$ is a \emph{quaternion algebra} over a field $F$---that is, a central simple algebra over $F$ such that every element has degree at most $2$ over $F$. For simplicity, we shall only consider the case where $\text{char}(F) \neq 2$, in which case we can equivalently describe a quaternion algebra as an $F$-algebra generated by two elements $i,j$ subject to the relations $i^2 = a, j^2 = b, ij = -ji$ for some $a,b \in F^\times$---we typically denote such an algebra by
    \begin{align*}
        \left(\frac{a,b}{F}\right).
    \end{align*}
    
\noindent It is easy to check that such an algebra is dimension $4$---each element can be written in the form $x + yi + zj + tij$ for some $x,y,z,t \in F$---and it has an involution $\overline{x + yi + zj + tij} = x - yi - zj - tij$ known as the \emph{standard involution} or \emph{quaternion conjugation}.The subspace on which the standard involution acts as the identity is just $F$; the subspace on which it acts as multiplication by $-1$ is three-dimensional and will be denoted by $H^0$. It is common to define the \emph{(reduced) trace} and $\emph{(reduced) norm}$ in terms of the standard involution as
    \begin{align*}
        \tr(x) &= x + \overline{x} \\
        \nrm(x) &= x\overline{x},
    \end{align*}
    
\noindent respectively. The standard involution is clearly an involution of the first kind. Any other involution of the first kind will be of the form
    \begin{align*}
        \sigma: H &\rightarrow H \\
        x &\mapsto a\overline{x}a^{-1}
    \end{align*}
    
\noindent for some $a \in H^\times \cap H^0$. One can check that quaternion conjugation is the unique symplectic involution on the quaternion algebra, whereas all the other involutions are orthogonal. In fact, one can show that for any orthogonal involution, if one correctly chooses a basis $1,i,j,ij$ for $H$, then the involution will be of the form
    \begin{align*}
        (x + yi + zj + tij)^\ddagger &= x + yi + zj - tij
    \end{align*}
	
\noindent Clearly, any such involution will act as the identity on a subspace of dimension $3$, which we shall denote by $H^+$, and act as multiplication by $-1$ on a subspace of dimension $1$, which we shall denote by $H^-$.

\begin{remark}
An objection might be raised to the notation $\ddagger$ to denote an orthogonal involution. The motivation for this notation is simple: $\Mat(2,F)$ is a quaternion algebra and one can check that in that case
    \begin{align*}
        \overline{\begin{pmatrix} a & b \\ c & d \end{pmatrix}} = \begin{pmatrix} d & -b \\ -c & a \end{pmatrix}.
    \end{align*}
    
\noindent This is the adjugate, which is usually denoted by $\dagger$. Since orthogonal involutions are related to the standard involution but are nevertheless distinct, we choose the notation $\ddagger$ to represent them.
\end{remark}

Since orthogonal involutions act as the identity on a space of dimension $1$, we can define their \emph{discriminant} as follows:
    \begin{align*}
        \disc: \left\{\text{orthogonal involutions}\right\} & \rightarrow F^\times / \left(F^{\times}\right)^2 \\
        \ddagger &\mapsto x^2 \left(F^{\times}\right)^2,
    \end{align*}

\noindent where $x$ is any element of $H^\times \cap H^-$. The discriminant uniquely determines the involution in the sense that if $\ddagger_1,\ddagger_2$ are orthogonal involutions on $H$ then there exists an isomorphism of rings with involutions $(H,\ddagger_1) \cong (H,\ddagger_2)$ if and only if $\disc(\ddagger_1) = \disc(\ddagger_2)$. In general, a quaternion algebra admits many inequivalent orthogonal involutions; however, the linear algebraic groups $SL^\ddagger(2,H)$ that arise as a result are all essentially the same.

\begin{lemma}\label{Conjugate groups}
Let $F$ be a field of characteristic not $2$. Let $H$ be a quaternion algebra over $F$, and let $\ddagger_1,\ddagger_2$ be orthogonal involutions on $H$. Then $SL^{\ddagger_1}(2,H)$ and $SL^{\ddagger_2}(2,H)$ are conjugate inside $GL(2,H \otimes_F \overline{F})$, where $\overline{F}$ is the algebraic closure of $F$.
\end{lemma}

\begin{proof}
Both $\ddagger_1$ and $\ddagger_2$ can be extended to orthogonal involution on $H \otimes_F \overline{F}$---however, since every element is a square in $\overline{F}$, it follows that $(H \otimes_F \overline{F}, \ddagger_1) \cong (H \otimes_F \overline(F), \ddagger_2)$. Since $H \otimes_F \overline{F}$ is a central simple algebra, by the Skolem-Noether theorem there must exist $u \in H \otimes_F \overline{F}$ such that the desired isomorphism of the form $x \mapsto uxu^{-1}$ for all $x \in H \otimes_F \overline{F}$. Now, choose any
    \begin{align*}
        \gamma = \begin{pmatrix} a & b \\ c & d \end{pmatrix} \in SL^{\ddagger_1}(2,H)
    \end{align*}
    
\noindent and note that
    \begin{align*}
        \left(\begin{pmatrix} u & 0 \\ 0 & u \end{pmatrix}\gamma\begin{pmatrix} u & 0 \\ 0 & u \end{pmatrix}^{-1}\right)^{\hat{\ddagger}_2} &= \begin{pmatrix} uau^{-1} & ubu^{-1} \\ ucu^{-1} & udu^{-1} \end{pmatrix}^{\ddagger_2} \\
        &= \begin{pmatrix} \left(udu^{-1}\right)^{\ddagger_2} & -\left(ubu^{-1}\right)^{\ddagger_2} \\ -\left(ucu^{-1}\right)^{\ddagger_2} & \left(uau^{-1}\right)^{\ddagger_2} \end{pmatrix} \\
        &= \begin{pmatrix} ud^{\ddagger_1}u^{-1} & -ub^{\ddagger_1}u^{-1} \\ -uc^{\ddagger_1}u^{-1} & ua^{\ddagger_1}u^{-1} \end{pmatrix} \\
        &= \begin{pmatrix} u & 0 \\ 0 & u \end{pmatrix}\gamma^{\hat{\ddagger}_1}\begin{pmatrix} u & 0 \\ 0 & u \end{pmatrix}^{-1} \\
        &= \begin{pmatrix} u & 0 \\ 0 & u \end{pmatrix}\gamma^{-1}\begin{pmatrix} u & 0 \\ 0 & u \end{pmatrix}^{-1},
    \end{align*}
    
\noindent from which we conclude that we have constructed a map
    \begin{align*}
        SL^{\ddagger_1}(2,H) &\rightarrow SL^{\ddagger_2}(2,H) \\
        \gamma &\mapsto \begin{pmatrix} u & 0 \\ 0 & u \end{pmatrix}\gamma\begin{pmatrix} u & 0 \\ 0 & u \end{pmatrix}^{-1}.
    \end{align*}
    
\noindent It is easy to see that this map is the desired isomorphism.
\end{proof}

\begin{remark}
Since these groups are conjugate inside $GL(2,H\otimes_F \overline{F})$, they are isomorphic as algebraic groups. In contrast, purely by dimensional considerations, we can see that $SL^\dagger(2,H)$ is not isomorphic to these algebraic groups. Furthermore, if we restrict from $H$ to a sub-ring, we shall again find that these groups are not isomorphic in general.
\end{remark}

We know that if $\ddagger$ is an orthogonal involution, then $SL^\ddagger(2,H)$ is a symplectic group. However, in low dimensions, there is an accidental isomorphism between symplectic groups and spin groups. With our machinery, we can work out this isomorphism very explicitly.

\begin{theorem}\label{Algebraic Group Exact Sequence}
Let $H$ be a quaternion algebra over a field $F$ not characteristic $2$, with orthogonal involution $\ddagger$. Define a quadratic form $q_H$ on $F^2 \oplus H^+$ by
	\begin{align*}
	q_H(s,t,z) = st - \nrm(z).
	\end{align*}
	
\noindent Then there is an exact sequence of algebraic groups
	\begin{align*}
	1 \rightarrow \left\{\pm 1\right\} \rightarrow SL^\ddagger(2,H) \rightarrow O^0(q_H) \rightarrow 1,
	\end{align*}
	
\noindent where $O^0(q_H)$ is the connected component of the orthogonal group of $q_H$.
\end{theorem}

\begin{remark}
The special case where $K = \QQ$ and $H$ is positive definite was worked out in \cite{Sheydvasser2019}. We follow mostly the same argument.
\end{remark}

\begin{proof}
Define a set
	\begin{align*}
	\mathcal{M}_H = \left\{M = \begin{pmatrix} a & b \\ c & d \end{pmatrix} \in \Mat(2,H)\middle|\overline{M}^T = M, \ ab^\ddagger, cd^\ddagger \in H^+\right\}.
	\end{align*}
	
\noindent It is easy to see that there is a bijective linear map
	\begin{align*}
	F^2 \oplus H^+ &\rightarrow \mathcal{M}_H \\
	(s, t, z) &\mapsto \begin{pmatrix} s & z \\ \overline{z} & t \end{pmatrix},
	\end{align*}
	
\noindent taking the quadratic norm to the quasi-determinant $st^\ddagger -z\overline{z}^\ddagger = st - \nrm(z)$---thus, we can identify these two sets. On the other hand, if $\gamma \in SL^\ddagger(2,H)$ and $M \in \mathcal{M}_H$, then it is easy to check that $\gamma M \overline{\gamma}^T \in \mathcal{M}_H$ as well, and has the same quasi-determinant as $\gamma$. Therefore, we have defined a morphism of algebraic groups $SL^\ddagger(2,H) \rightarrow O(q_H)$. For any element of the kernel,
	\begin{align*}
	\begin{pmatrix} 1 & 0 \\ 0 & 0 \end{pmatrix} &= \begin{pmatrix} a & b \\ c & d \end{pmatrix}\begin{pmatrix} 1 & 0 \\ 0 & 0 \end{pmatrix}\begin{pmatrix} \overline{a} & \overline{c} \\ \overline{b} & \overline{d} \end{pmatrix} \\
	&= \begin{pmatrix} a & 0 \\ c & 0 \end{pmatrix}\begin{pmatrix} \overline{a} & \overline{c} \\ \overline{b} & \overline{d} \end{pmatrix} \\
	&= \begin{pmatrix} \nrm(a) & a\overline{c} \\ c\overline{a} & \nrm(c) \end{pmatrix},
	\end{align*}

\noindent from which we conclude that $c = 0$ and $\nrm(a) = 1$. Similarly, the relation
	\begin{align*}
	\begin{pmatrix} 0 & 0 \\ 0 & 1 \end{pmatrix} &= \begin{pmatrix} a & b \\ 0 & d \end{pmatrix}\begin{pmatrix} 0 & 0 \\ 0 & 1 \end{pmatrix}\begin{pmatrix} \overline{a} & 0 \\ \overline{b} & \overline{d} \end{pmatrix} \\
	&= \begin{pmatrix} 0 & b \\ 0 & d \end{pmatrix}\begin{pmatrix} \overline{a} & 0 \\ \overline{b} & \overline{d} \end{pmatrix} \\
	&= \begin{pmatrix} \nrm(b) & b\overline{d} \\ d\overline{b} & \nrm(d) \end{pmatrix}
	\end{align*}
	
\noindent gives us that $b = 0$ and $\nrm(d) = 1$. Finally, we note that
	\begin{align*}
	\begin{pmatrix} 0 & z \\ \overline{z} & 0 \end{pmatrix} &= \begin{pmatrix} a & 0 \\ 0 & d \end{pmatrix}\begin{pmatrix} 0 & z \\ \overline{z} & 0 \end{pmatrix}\begin{pmatrix} \overline{a} & 0 \\ 0 & \overline{d} \end{pmatrix} \\
	&= \begin{pmatrix} 0 & az \\ d\overline{z} & 0 \end{pmatrix}\begin{pmatrix} \overline{a} & 0 \\ 0 & \overline{d} \end{pmatrix} \\
	&= \begin{pmatrix} 0 & az\overline{d} \\ d\overline{z}\,\overline{a} & 0 \end{pmatrix}
	\end{align*}
	
\noindent implies $az\overline{d} = z$ for all $z \in H^+$. Since $\nrm(d) = 1$, this is just to say that $az = zd$ for all $z \in H^+$, and since $ad^\ddagger = 1$, this is the same as saying that $az = z\overline{a^\ddagger}$ for all $z \in H^+$. It is easy to check this equation is satisfied only if $a \in F$, but since $\nrm(a) = 1$, we see that $a^2 = 1$, and therefore the kernel actually just consists of $\pm 1$, as claimed. Since the kernel has dimension $0$, the dimension of the image is $\dim\left(SL^\ddagger(2,H)\right) = 10$ by Corollary \ref{dimension computation}, which is the same as the dimension of $O(q_H)$. Since $SL^\ddagger(2,H)$ is a symplectic group by Corollary \ref{Is symplectic or orthogonal}, it is connected, and so its image must be $O^0(q_H)$, the connected component of the identity.
\end{proof}

In characteristic $0$, this sets up a correspondence between the groups $SL^\ddagger(2,H)$ and spin groups of quadratic forms.

\begin{theorem}\label{Correspondence between QAs and Spin Groups}
Let $F$ be a characteristic $0$ field. Then there is a bijection
	\begin{align*}
	\left\{\substack{\text{Isomorphism classes of} \\ \text{quaternion algebras over } F}\right\} &\rightarrow \left\{\substack{\text{Isomorphism classes of} \\ \text{spin groups of indefinite,} \\ \text{quinary quadratic forms over } F} \right\} \\
	[H] & \mapsto \left[SL^\ddagger(2,H)\right].
	\end{align*}
\end{theorem}

\begin{proof}
By Theorem \ref{Algebraic Group Exact Sequence}, we know that $SL^\ddagger(2,H)$ is a double-cover of an orthogonal group, which is to say that it is a spin group. By Lemma \ref{Conjugate groups}, we know that this map is well-defined in that the choice of orthogonal involution $\ddagger$ does not change the isomorphism class of $SL^\ddagger(2,H)$. It is easy to check that this map is surjective. Choose any indefinite, quinary quadratic form $q$ over $F$. Since it is indefinite, we can decompose it as $\langle 1, -1 \rangle \oplus \langle a,b,c \rangle$, for some $a,b,c \in F^\times$. In fact, since scaling the quadratic form does not change the spin group, we can assume that the quadratic form is $\langle 1, -1 \rangle \oplus \langle 1,b,c \rangle$. In that case, it is clear that the image of
	\begin{align*}
	H = \left(\frac{-b,-c}{F}\right)
	\end{align*}
	
\noindent will be the desired spin group. So, we are finally left with checking that the map is injective, which is to say that if $SL^{\ddagger_1}(2,H_1)$ is isomorphic to $SL^{\ddagger_2}(2,H_2)$, then $H_1 \cong H_2$. An isomorphism of the algebraic groups induces an isomorphism of the Lie algebras, which we know are
    \begin{align*}
        \mathfrak{sl}^{\ddagger_i}(2,H_i) = \left\{M \in \Mat(2,H_i)\middle| M^{\hat{\ddagger}_i} = -M\right\}
    \end{align*}
    
\noindent by Theorem \ref{Lie algebra computation}. This extends to an isomorphism
    \begin{align*}
        \mathfrak{sl}^{\ddagger_1}(2,H_1 \otimes_F \overline{F}) \cong \mathfrak{sl}^{\ddagger_2}(2,H_2 \otimes_F \overline{F}).
    \end{align*}
    
\noindent However, we know that $(H_1 \otimes_F \overline{F}, \ddagger_1) \cong (H_2 \otimes_F \overline{F}, \ddagger_2)$, and in fact we can view $(H_1,\ddagger_1)$ and $(H_2, \ddagger_2)$ as embedded inside of a ring with involution $(H,\ddagger)$ where $H$ is the unique quaternion algebra over $\overline{F}$. Thus, the isomorphism of Lie algebras can be considered as coming from an automorphism of $\mathfrak{sl}^\ddagger(2,H)$. However, by Theorem \ref{dimension restriction}, we know that this is isomorphic to $\mathfrak{sp}(4)$ which is a simple Lie algebra of type $C_n$. In characteristic $0$, the outer automorphisms of such Lie algebras correspond to graph automorphisms of their corresponding Dynkin diagrams---however, there are no such automorphisms for the $C_n$ type, and therefore the automorphism of $\mathfrak{sl}^\ddagger(2,H)$ must be inner. That is to say, there exists some element $\gamma \in SL^\ddagger(2,H)$ such that
    \begin{align*}
        \mathfrak{sl}^{\ddagger_1}(2,H_1) &\rightarrow \mathfrak{sl}^{\ddagger_2}(2,H_2) \\
        M &\mapsto \gamma M \gamma^{-1}.
    \end{align*}

\noindent This extends to an isomorphism of $\Mat(2,H_1)$ with $\Mat(2, H_2)$. Why is this? Well, any element $M \in \Mat(2,H_1)$ can be written uniquely as a sum $M = M' + M''$ such that ${M'}^{\hat{\ddagger}} = M'$ and ${M''}^{\hat{\ddagger}} = -M''$. Clearly, $M'' \in \mathfrak{sl}^{\ddagger_1}(2,H_1)$ and there exists some $M''' \in \mathfrak{sl}^{\ddagger_1}(2,H_1)$ such that
    \begin{align*}
        M' = \underbrace{\begin{pmatrix} 1 & 0 \\ 0 & -1 \end{pmatrix}}_{\in \mathfrak{sl}^{\ddagger_1}(2,H_1)}M'''.
    \end{align*}
    
\noindent Therefore, $\gamma M' \gamma^{-1}, \gamma M'' \gamma^{-1} \in \Mat(2,H_2)$, so $\gamma M \gamma^{-1} \in \Mat(2,H_2)$ for all $M \in \Mat(2,H_1)$, and so we have our desired map
    \begin{align*}
        \Mat(2,H_1) &\rightarrow \Mat(2,H_2) \\
        M &\mapsto \gamma M \gamma^{-1}.
    \end{align*}
    
\noindent However, $\Mat(2,H_1)$ and $\Mat(2,H_2)$ are central simple algebras and so there is an isomorphism between them if and only if $H_1 \cong H_2$.
\end{proof}

\begin{remark}
Theorem \ref{Correspondence between QAs and Orthogonal Groups} is an immediate consequence of Theorem \ref{Correspondence between QAs and Spin Groups}.
\end{remark}

\section{Orders of Quaternion Algebras:}\label{SECTION Quaternion Orders}

We shall now consider some of the special features that are true for orders of quaternion algebras. First, note that all orders of a quaternion algebra are automatically closed under quaternion conjugation; therefore, there is only interest in looking at orders closed under an orthogonal involution. We previously noted that orthogonal involutions are classified by their discriminant. As it happens, if $\ddagger$ is an orthogonal involution on a quaternion algebra, then we can compute an important algebraic invariant of the maximal $\ddagger$-orders in terms of this discriminant. However, to state our desired result, we shall need two other notions of discriminant as well. First, any quaternion algebra $H$ over $K$ is either a division algebra or isomorphic to $\Mat(2,K)$. For any place $\nu$ of $K$, we say that $H$ \emph{ramifies} if $H_\nu$ is a division algebra, and we say that it \emph{splits} otherwise. Recalling that the finite places of $K$ correspond to its prime ideals, we define the \emph{discriminant} of $H$ to be the ideal
    \begin{align*}
        \disc(H) = \prod_{\substack{\mathfrak{p} \text{ a prime ideal} \\ H_\mathfrak{p} \text{ ramifies}}} \mathfrak{p}.
    \end{align*}
    
\noindent The places at which $H$ ramify uniquely determine it up to isomorphism; thus, knowing the infinite places where $H$ ramifies and the discriminant uniquely determines $H$. In fact, since the number of ramified places is always even, over $\QQ$ the discriminant uniquely determines the isomorphism class. There is also the related notion of the discriminant of an order $\OO$, which we shall define as
    \begin{align*}
        \disc(\OO)^2 = \det\left(\tr(e_i \overline{e_j})\right)_{0 \leq i,j \leq 3}\mathfrak{o}_K,
    \end{align*}
    
\noindent where $e_0, e_1, e_2, e_3$ is any basis of $\OO$. One checks that the expression on the right is always a square ideal. One also checks that $\OO$ is a maximal order if and only if $\disc(\OO) = \disc(H)$. A similar characterization applies to $\ddagger$-orders as well.

\begin{theorem}\cite[Theorem 1.1]{Sheydvasser2017}\label{MainTheoremOrderPaper}
Given a quaternion algebra $H$ over a local or global field $F$ of characteristic not $2$ and with an orthogonal involution $\ddagger$, the maximal $\ddagger$-orders of $H$ are exactly the orders of the form $\OO \cap \OO^\ddagger$ with discriminant
	\begin{align*}
	\disc(H) \cap \iota(\disc(\ddagger)),
	\end{align*}
	
\noindent where $\OO$ is a maximal order and $\iota$ is the map
	\begin{align*}
	\iota: F^\times/\left(F^\times\right)^2 &\rightarrow \left\{\text{square-free ideals of } \mathfrak{o}_F\right\} \\
	[\lambda] &\mapsto \bigcup_{\lambda \in [\lambda] \cap \mathfrak{o}} \lambda \mathfrak{o}_F.
	\end{align*}
\end{theorem}

\noindent Over local fields, one can obtain more detailed information about isomorphism classes. Surprisingly, unlike maximal orders, maximal $\ddagger$-orders are not necessarily all of the same isomorphism class over a local field; if the maximal ideal of the ring of integers contains $2$, there can be multiple isomorphism classes---precise statements can be found in \cite{Sheydvasser2017}. Depending on the choice of involution, maximal $\ddagger$-orders can be maximal (in the usual sense) or strictly smaller. The number of isomorphism classes depends both on the class number and the number of local isomorphism classes.

\begin{example}\label{Non Isomorphic Orders}
Define $(x + yi + zj + tij)^\ddagger = x + yi - zj + tij$. Then
	\begin{align*}
	\OO_1 &= \ZZ \oplus \ZZ i \oplus \ZZ \frac{1 + j}{2} \oplus \ZZ \frac{i + ij}{2} \subset \left(\frac{-1,-23}{\QQ}\right) \\
	\OO_2 &= \ZZ \oplus 3\ZZ i \oplus \ZZ \frac{1 + j}{2} \oplus \ZZ \frac{11 i + ij}{6} \subset \left(\frac{-1,-23}{\QQ}\right)
	\end{align*}
	
\noindent are both easily checked to be $\ddagger$-orders. Both of them have discriminant $(23)$, which is the discriminant of the quaternion algebra; consequently, they are both maximal and $\ddagger$-maximal. All of the localizations of $\OO_1$ and $\OO_2$ are isomorphic as algebras with involution; this can be seen from the fact that there is only one isomorphism class for each localization \cite{Sheydvasser2017}. However, not only is it true that $(\OO_1,\ddagger) \ncong (\OO_2,\ddagger)$, in fact $\OO_1 \ncong \OO_2$, since $\OO_1^\times = \{1,-1,i,-i\}$, whereas $\OO_2^\times = \{1,-1\}$. This is an easy computation using the fact that $u \in \OO_i^\times$ if and only if $\nrm(u) =  1$; since both $\OO_i$ are contained inside a definite quaternion algebra, there are only finitely many possibilities for the units and it is easy to enumerate all of them.
\end{example}

\begin{example}
Define $(x + yi + zj + tij)^\ddagger = x + yi + zj - tij$. Then
	\begin{align*}
	\OO_1 &= \ZZ \oplus \ZZ i \oplus \ZZ \frac{i + j}{2} \oplus \ZZ \frac{1 + ij}{2} \subset \left(\frac{-1,-3}{\QQ}\right) \\
	\OO_2 &= \ZZ \oplus \ZZ i \oplus \ZZ \frac{1 + j}{2} \oplus \ZZ \frac{i + ij}{2} \subset \left(\frac{-1,-3}{\QQ}\right)
	\end{align*}
	
\noindent are both $\ddagger$-orders, and both have discriminant $(3)$. Thus, they are both maximal and $\ddagger$-maximal. It is true that $\OO_1 \cong \OO_2$---this can be seen from a computation of the class number, which is $1$---but $(\OO_1,\ddagger) \ncong (\OO_2,\ddagger)$. This is because $\tr(\OO_1^+) = (2)$, but $\tr(\OO_2^+) = (1)$.
\end{example}

\begin{example}
Define $(x + yi + zj + tij)^\ddagger = x + yi + zj - tij$. Then
    \begin{align*}
        \OO = \ZZ \oplus \ZZ i \oplus \ZZ \frac{1 + i + j}{2} \oplus \ZZ \frac{1 + i + ij}{2} \subset \left(\frac{-1,-6}{\QQ}\right)
    \end{align*}
    
\noindent is a $\ddagger$-order with discriminant $(6)$. The discriminant of the quaternion algebra is $3$, so this order is not maximal. However, it is $\ddagger$-maximal since $\disc(\ddagger) = -6 \left(\QQ^\times\right)^2$, and therefore $\disc(H) \cap \iota(\disc(\ddagger)) = (6)$.
\end{example}

We already know that the groups $SL^\ddagger(2,\OO)$ are arithmetic subgroups of symplectic groups or, equivalently by Theorem \ref{Algebraic Group Exact Sequence}, arithmetic subgroups of spin groups. The case of the greatest interest to the author is when $H$ is a definite, rational quaternion algebra as then $SL^\ddagger(2,\OO)$ will be an arithmetic subgroup of $\Isom^0(\HH^4) \cong SO^+(4,1)$. However, we can state our result a little more generally and just consider the case where $H$ is a rational quaternion algebra; then $SL^\ddagger(2,\OO)$ will either be an arithmetic group of $SO^+(4,1)$ if $H$ is definite or $SO^+(3,2)$ if it is indefinite. In either case, the matrix ring is now an invariant under group isomorphism.

\begin{lemma}\label{Ring is Group Invariant}
Let $H_1, H_2$ be rational quaternion algebras with orthogonal involutions $\ddagger_1, \ddagger_2$. Let $\Gamma_1, \Gamma_2$ be lattices of $SL^{\ddagger_1}(2,H_1)$, $SL^{\ddagger_2}(2,H_2)$ such that their centers are $\{\pm 1\}$. If $\Gamma_1$ and $\Gamma_2$ are isomorphic as groups, then $(\ZZ[\Gamma_1],\hat{\ddagger}_1) \cong (\ZZ[\Gamma_2],\hat{\ddagger}_2)$.
\end{lemma}

\begin{proof}
Let $\phi: \Gamma_1 \rightarrow \Gamma_2$ be the group isomorphism. Note that $\phi(-I) = -I$, since $-I$ is the unique non-identity element of the centers of $\Gamma_i$. Therefore, it induces an isomorphism $\overline{\phi}: \overline{\Gamma}_1 \rightarrow \overline{\Gamma}_2$ between the images of the $\Gamma_i$ inside $SL^{\ddagger_i}(2,H_i)/\{\pm I\}$. Note that
	\begin{align*}
	SL^{\ddagger_i}(2,H_i \otimes_\QQ \RR)/\{\pm I\} \cong \begin{cases} PSO^+(4,1) & \text{if } H_i \otimes_\QQ \RR \cong H_\RR \\ PSO^+(3,2) & \text{if } H_i \otimes_\QQ \RR \cong \Mat(2,\RR), \end{cases}
	\end{align*}

\noindent and so we can apply the Mostow rigidity theorem to conclude that both $\overline{\Gamma_1}$ and $\overline{\Gamma_2}$ can be viewed as being lattices of the same Lie group $G$, namely either $PSO^+(3,2)$ or $PSO^+(4,1)$. In either case, $G$ is simple, and therefore by Mostow rigidity $\overline{\Gamma_1}$ and $\overline{\Gamma_2}$ are conjugate in $G$. Let $g \in G$ be such that $\overline{\Gamma_2} = g \overline{\Gamma_1}g^{-1}$, and choose an element $g' \in SL^\ddagger(2,H_i \otimes_\RR \RR)$ such that its image in the quotient is $g$. Then we have a well-defined ring isomorphism
	\begin{align*}
	\Phi: \ZZ[\Gamma_1] &\rightarrow \ZZ[\Gamma_2] \\
	M &\mapsto g' M {g'}^{-1}.
	\end{align*}
	
\noindent Letting $\ddagger$ be the orthogonal involution on $H_1 \otimes_\QQ \RR \cong H_2 \otimes_\QQ \RR$, we see that
    \begin{align*}
        \left(g' M {g'}^{-1}\right)^{\hat{\ddagger}} = g' M^{\hat{\ddagger}} {g'}^{-1}
    \end{align*}
    
\noindent since $g' \in SL^\ddagger(2,H_i \otimes_\RR \RR)$. Therefore, the map we have constructed is an isomorphism of rings with involution.
\end{proof}

\begin{remark}
The use of the Mostow rigidity theorem in Lemma \ref{Ring is Group Invariant} makes clear why we restrict to the case where $H$ is a quaternion algebra over $\QQ$---over other number fields, the set of infinite places $\Omega_\infty$ has more than one element, and so rather than $SL^\ddagger(2,\OO)/\{\pm I\}$ injecting as a lattice into $PSO^0(4,1)$ or $PSO^0(3,2)$, it will instead inject into some non-simple Lie group.
\end{remark}

\begin{remark}
In light of Theorem \ref{Conjugation turns into isomorphism}, Lemma \ref{Ring is Group Invariant} immediately implies Theorem \ref{special isomorphism theorem}.
\end{remark}

For commutative rings, it is true that $\Mat(2,R) \cong \Mat(2,S)$ implies $R \cong S$. This is false in general for non-commutative rings. We know that $SL^{\ddagger_1}(2,\OO_1) \cong SL^{\ddagger_2}(2,\OO_2)$ if and only if $(\Mat(2,\OO_1),\hat{\ddagger}_1) \cong (\Mat(2,\OO_2),\hat{\ddagger}_1)$; we now give a couple of examples showing that this latter condition could not be replaced with $(\OO_1, \ddagger_1) \cong (\OO_2, \ddagger_2)$ or $\Mat(2,\OO_1) \cong \Mat(2,\OO_2)$.

\begin{example}\label{Isomorphic but not Conjugate Groups}
Let
	\begin{align*}
	\OO_1 &= \ZZ \oplus \ZZ i \oplus \ZZ \frac{1 + j}{2} \oplus \ZZ \frac{i + ij}{2} \subset \left(\frac{-1,-7}{\QQ}\right) \\
	\OO_2 &= \ZZ \oplus \ZZ i \oplus \ZZ \frac{i + j}{2} \oplus \ZZ \frac{1 + ij}{2} \subset \left(\frac{-1,-7}{\QQ}\right).
	\end{align*}
	
\noindent Both of these are maximal $\ddagger$-orders, if we take the usual involution $\left(x + yi + zj + tij\right)^\ddagger = x + yi + zj - tij$. Since $\tr(\OO_1 \cap H^+) = (1)$ and $\tr(\OO_2 \cap H^+) = (2)$, we see that $(\OO_1,\ddagger) \ncong (\OO_2,\ddagger)$. However, $(1 + i) \OO_1 (1 + i)^{-1} = \OO_2$, and therefore
	\begin{align*}
	\begin{pmatrix} \frac{1 + i}{\sqrt{2}} & 0 \\ 0 & \frac{-1 + i}{\sqrt{2}} \end{pmatrix} SL^\ddagger(2,\OO_1) \begin{pmatrix} \frac{1 + i}{\sqrt{2}} & 0 \\ 0 & \frac{-1 + i}{\sqrt{2}} \end{pmatrix}^{-1} &= SL^\ddagger(2,\OO_2).
	\end{align*}
\end{example}

\begin{remark}
This example was originally worked out in \cite{Sheydvasser2019}.
\end{remark}

\begin{example}\label{Non Isomorphic Orders with Isomorphic Groups}
Take $\OO_1, \OO_2$ as in Example \ref{Non Isomorphic Orders}. We proved that $\OO_1 \ncong \OO_2$; however, we claim that $SL^\ddagger(2,\OO_1) \cong SL^\ddagger(2,\OO_2)$.
\end{example}

\begin{proof}
Define
	\begin{align*}
	\gamma = \begin{pmatrix} \frac{1 - 6i + j}{2\sqrt{3}} & \frac{1 + j}{2\sqrt{3}} \\ \frac{1 + 6i + j}{2\sqrt{3}} & \sqrt{3} i \end{pmatrix} \in SL^\ddagger\left(2,H \otimes_\QQ \QQ(\sqrt{3})\right).
	\end{align*}
	
\noindent By an easy computation,

	\begin{minipage}{.45\textwidth}
	\begin{align*}
	\gamma \begin{pmatrix} 0 & 1 \\ 0 & 0 \end{pmatrix} \gamma^{-1} &\in \Mat(2,\OO_2) \\
	\gamma \begin{pmatrix} 0 & i \\ 0 & 0 \end{pmatrix} \gamma^{-1} &\in \Mat(2,\OO_2) \\
	\gamma \begin{pmatrix} 0 & \frac{1 + j}{2} \\ 0 & 0 \end{pmatrix} \gamma^{-1} &\in \Mat(2,\OO_2)
	\end{align*}
	\end{minipage}%
	\begin{minipage}{.45\textwidth}
	\begin{align*}
	\gamma \begin{pmatrix} 0 & \frac{i + ij}{2} \\ 0 & 0 \end{pmatrix} \gamma^{-1} &\in \Mat(2,\OO_2) \\
	\gamma \begin{pmatrix} 0 & 0 \\ 1 & 0 \end{pmatrix} \gamma^{-1} &\in \Mat(2,\OO_2).
	\end{align*}
	\end{minipage}
	
\noindent However, the elements
    \begin{align*}
        \begin{pmatrix} 0 & 1 \\ 0 & 0 \end{pmatrix}, \begin{pmatrix} 0 & i \\ 0 & 0 \end{pmatrix}, \begin{pmatrix} 0 & \frac{1 + j}{2} \\ 0 & 0 \end{pmatrix}, \begin{pmatrix} 0 & \frac{i + ij}{2} \\ 0 & 0 \end{pmatrix}, \begin{pmatrix} 0 & 0 \\ 1 & 0 \end{pmatrix}
    \end{align*}
	
\noindent generate $\Mat(2,\OO_1)$ as a $\ZZ$-algebra, so we have a well-defined, injective ring homomorphism
	\begin{align*}
	\Mat(2,\OO_1) &\rightarrow \Mat(2,\OO_2) \\
	M &\mapsto \gamma M \gamma^{-1}.
	\end{align*}
	
\noindent This map must be surjective, as can be checked either from computing the discriminants of $\Mat(2,\OO_1)$ and $\Mat(2,\OO_2)$, or simply by noting that $SL^\ddagger(2,\OO_1), SL^\ddagger(2,\OO_2)$ are maximal arithmetic groups. In any case, this is an isomorphism of rings with involution, and therefore by Theorem \ref{special isomorphism theorem}, $SL^\ddagger(2,\OO_1) \cong SL^\ddagger(2,\OO_2)$.
\end{proof}

\begin{remark}
Note that we have effectively produced a proof that one can find two non-isomorphic rings $\OO_1, \OO_2$ such that $\Mat(2,\OO_1) \cong \Mat(2,\OO_2)$. While this is technically a new proof of that fact, this example was in fact already worked out by Chatters \cite[Example 5.1]{Chatters1996}.
\end{remark}

\begin{example}
Let
	\begin{align*}
	\OO_1 &= \ZZ \oplus \ZZ i \oplus \ZZ j \oplus \ZZ \frac{1 + i + j + ij}{2} \subset \left(\frac{-1,-5}{\QQ}\right) \\
	\OO_2 &= \ZZ \oplus \ZZ i \oplus \ZZ \frac{1 + i + j}{2} \oplus \ZZ \frac{1 + i + ij}{2} \subset \left(\frac{-1,-10}{\QQ}\right).
	\end{align*}
	
\noindent In both cases, define $(x + yi + zj + tij)^\ddagger = x + yi + zj - tij$---this is slight abuse of notation, since these two orders have entirely different bases. However, $\OO_1 \cong \OO_2$---this is because both of their ambient quaternion algebras have discriminant $(2)$ and so are isomorphic, they both have discriminant $(10)$, and their class numbers are $1$. Thus, they are both Eichler orders of the same level and must be conjugate to one another. This means that $\Mat(2,\OO_1) \cong \Mat(2,\OO_2)$. However, we claim that $(\Mat(2,\OO_1),\hat{\ddagger}) \ncong (\Mat(2,\OO_2),\hat{\ddagger})$, and consequently $SL^\ddagger(2,\OO_1) \ncong SL^\ddagger(2,\OO_2)$.
\end{example}

\begin{proof}
Define two lattices
    \begin{align*}
        M_i = \left\{M \in \Mat(2,\OO_i)^+\middle|\tr(M) = 0\right\}
    \end{align*}

\noindent and consider the integral quadratic forms
    \begin{align*}
        q_i: M_i &\rightarrow \ZZ \\
        M &\mapsto \tr(M^2).
    \end{align*}
    
\noindent What do these quadratic forms look like? First, note that
    \begin{align*}
        \Mat(2,\OO_i)^+ &= \left\{\begin{pmatrix} a & b \\ c & a^\ddagger\end{pmatrix} \in \Mat(2,\OO_i)\middle| b,c \in \OO_i^-\right\} \\
        M_i &= \left\{\begin{pmatrix} a & b \\ c & a^\ddagger\end{pmatrix} \in \Mat(2,\OO_i)\middle| b,c \in \OO_i^-, \ \tr(a) = 0\right\}.
    \end{align*}
    
\noindent Furthermore,
    \begin{align*}
        \tr\left(\begin{pmatrix} a & b \\ c & a^\ddagger\end{pmatrix}^2\right) &= \tr\left(\begin{pmatrix} a^2 + bc & ab + ba^\ddagger \\ cb + a^\ddagger c & cb + \left(a^\ddagger\right)^2\end{pmatrix}\right) \\
        &= 2\tr(bc) + 2\tr(a^2).
    \end{align*}
    
\noindent Since $\tr(a) = 0$, $a^2 = -\nrm(a)$. Furthermore, any element $x \in \OO_i^-$ with be of the form $n ij$ for some integer $n$. Therefore, our quadratic forms actually look like $4\nrm(ij)st - 4\nrm(a)$. Working out exactly what this is in coordinates, we have
    \begin{align*}
        q_1(s,t,x,y,z) &= 5 s t - x^2 - 5 y^2 - x z - 5 y z - 3 z^2 \\
        q_2(s,t,x,y,z) &= 10 s t - x^2 - x y - 3 y^2 - x z - y z - 3 z^2.
    \end{align*}
    
\noindent One can check that these quadratic forms are not equivalent. However, if there was an isomorphism $(\Mat(2,\OO_1),\hat{\ddagger}) \rightarrow (\Mat(2,\OO_2),\hat{\ddagger})$ then it would give a polynomial map between $M_1$ and $M_2$, and thus an equivalence between $q_1$ and $q_2$.
\end{proof}

We include one final example showing that conjugacy in $SL^\ddagger(2,H \otimes_K \overline{K})$ cannot be replaced with conjugacy in $SL^\ddagger(2,H)$, even if $(\OO_1, \ddagger) \cong (\OO_2, \ddagger)$.

\begin{example}
Let $K$ be any characteristic $0$ local field with maximal ideal $\mathfrak{p}$. Choose any $\lambda \in \mathfrak{o}_F$ such that $1 + \lambda \in \mathfrak{p} \backslash \mathfrak{p}^2$. Then $H = \Mat(2,K)$ is a quaternion algebra over $F$, and
	\begin{align*}
	\begin{pmatrix} a & b \\ c & d \end{pmatrix}^\ddagger = \begin{pmatrix} a & c/\lambda \\ b\lambda & d \end{pmatrix}
	\end{align*}
	
\noindent defines an orthogonal involution. Since $\lambda \in \mathfrak{o}^\times$, $\OO_1 = \Mat(2,\mathfrak{o}_K)$ is a maximal $\ddagger$-order. Similarly,
	\begin{align*}
	\OO_2 = \underbrace{\begin{pmatrix} 1 & -1 \\ \lambda & 1 \end{pmatrix}}_{:= u} \OO_1 \begin{pmatrix} 1 & -1 \\ \lambda & 1 \end{pmatrix}^{-1}
	\end{align*}
	
\noindent must be a maximal $\ddagger$-order; indeed, $(\OO_1,\ddagger) \cong (\OO_2,\ddagger)$. However, $SL^\ddagger(2,\OO_1)$ and $SL^\ddagger(2,\OO_2)$ are not conjugate in $SL^\ddagger(2,H)$.
\end{example}

\begin{proof}
Suppose that there exist $a,b,c,d \in H$ such that
	\begin{align*}
	\gamma := \begin{pmatrix} a & b \\ c & d \end{pmatrix} \in SL^\ddagger(2,H)
	\end{align*}
	
\noindent and $\gamma SL^\ddagger(2,\OO_1)\gamma^{-1} = SL^\ddagger(2,\OO_2)$. Since
	\begin{align*}
	\begin{pmatrix} a & b \\ c & d \end{pmatrix}\begin{pmatrix} 1 & z \\ 0 & 1 \end{pmatrix}\begin{pmatrix} a & b \\ c & d \end{pmatrix}^{-1} &= \begin{pmatrix} * & aza^\ddagger \\ -czc^\ddagger & * \end{pmatrix} \\
	\begin{pmatrix} a & b \\ c & d \end{pmatrix}\begin{pmatrix} 1 & 0 \\ z & 1 \end{pmatrix}\begin{pmatrix} a & b \\ c & d \end{pmatrix}^{-1} &= \begin{pmatrix} * & -bzb^\ddagger \\ dzd^\ddagger & * \end{pmatrix},
	\end{align*}
	
\noindent we wish to determine for which $v \in H$ $v\OO_1^+ v^\ddagger \subset \OO_2$. This is the same as determining all $v \in H$ such that $u^{-1}v\OO_1^+ v^\ddagger u \subset \OO_1$. We shall show that this is possible only if $v \in \OO_1$, proving that $\gamma \in SL^\ddagger(2,\OO_1)$. Write
	\begin{align*}
	v = \begin{pmatrix} v_1 & v_2 \\ v_3 & v_4 \end{pmatrix},
	\end{align*}
	
\noindent so
	\begin{align*}
	u^{-1}v \begin{pmatrix} 1 & 0 \\ 0 & 0 \end{pmatrix} v^\ddagger u &= \begin{pmatrix}
 \frac{\left(v_1+v_3\right)^2}{\lambda +1} & -\frac{\left(\lambda  v_1-v_3\right) \left(v_1+v_3\right)}{\lambda (\lambda +1)} \\
 -\frac{\left(\lambda  v_1-v_3\right) \left(v_1+v_3\right)}{\lambda +1} & \frac{\left(\lambda  v_1-v_3\right)^2}{\lambda  (\lambda +1)} \end{pmatrix} \in \Mat(2,\mathfrak{o}_K) \\
	u^{-1}v \begin{pmatrix} 0 & 0 \\ 0 & 1 \end{pmatrix} v^\ddagger u &= \begin{pmatrix} \frac{\lambda  \left(v_2+v_4\right)^2}{\lambda +1} & -\frac{\left(\lambda  v_2-v_4\right) \left(v_2+v_4\right)}{\lambda +1} \\ -\frac{\lambda  \left(\lambda  v_2-v_4\right) \left(v_2+v_4\right)}{\lambda +1} & \frac{\left(\lambda  v_2-v_4\right)^2}{\lambda +1} \end{pmatrix} \in \Mat(2,\mathfrak{o}_K).
	\end{align*}
	
\noindent Note that this is only possible if $v_1 + v_3, \lambda v_1 - v_3, v_2 + v_4, \lambda v_2 - v_4 \in \mathfrak{p}$. From it follows that $(1 + \lambda)v_1,(1 + \lambda)v_3 \in \mathfrak{p}$, hence $v_1, v_3 \in \mathfrak{o}_F$. Therefore, $v_3, v_4 \in \mathfrak{o}_F$. We conclude that $v \in \OO_1$. Ergo, $\gamma \in SL^\ddagger(2,\OO_1)$, and so $\gamma SL^\ddagger(2,\OO_1) \gamma^{-1} = SL^\ddagger(2,\OO_1) = SL^\ddagger(2,\OO_2)$. However, since
	\begin{align*}
	u \begin{pmatrix} 1 & 0 \\ 0 & 0 \end{pmatrix} u^{-1} = \begin{pmatrix} \frac{1}{1+\lambda } & \frac{1}{1 + \lambda} \\ \frac{\lambda }{1+\lambda } & \frac{\lambda }{1 + \lambda} \end{pmatrix} \notin \OO_1,
	\end{align*}
	
\noindent $\OO_1 \neq \OO_2$, and so $SL^\ddagger(2,\OO_1) \neq SL^\ddagger(2,\OO_2)$. This is a contradiction, and so we are done.
\end{proof}

\bibliography{AlgebraicIsomorphism}
\bibliographystyle{alpha}
\end{document}